\documentclass[a4paper,10pt]{amsart}

\usepackage{amsthm}                   
\usepackage{amsmath}                  
\usepackage{amsfonts}                 
\usepackage{amssymb}                  
\usepackage{mathrsfs}                 
\usepackage{mathtools}								

\theoremstyle{definition}
\newtheorem{definition}{Definition}[section]
\newtheorem{remark}[definition]{Remark}
\newtheorem{example}{Example}

\theoremstyle{plain}
\newtheorem{theorem}[definition]{Theorem}
\newtheorem{lemma}[definition]{Lemma}
\newtheorem{corollary}[definition]{Corollary}
\newtheorem{proposition}[definition]{Proposition}

\numberwithin{equation}{section}

\newcommand{\R}{\ensuremath{\mathbb{R}}}     
\newcommand{\N}{\ensuremath{\mathbb{N}}}     
\newcommand{\Z}{\ensuremath{\mathbb{Z}}}     

\def\sst{\scriptstyle}
\def\lt{\left}
\def\rt{\right}

\DeclareFontFamily{OMX} {MnSymbolE}{}

\DeclareFontShape{OMX}{MnSymbolE}{m}{n}{
    <-6>  MnSymbolE5
   <6-7>  MnSymbolE6
   <7-8>  MnSymbolE7
   <8-9>  MnSymbolE8
   <9-10> MnSymbolE9
  <10-12> MnSymbolE10
  <12->   MnSymbolE12}{}

\DeclareSymbolFont{largesymbol}  {OMX}{MnSymbolE}{m}{n}

\DeclareMathSymbol{\llangle}{\mathopen}{largesymbol}{'164}
\DeclareMathSymbol{\rrangle}{\mathclose}{largesymbol}{'171}

\newcommand{\set}[1]{\lt\{#1\rt\}}                   									
\newcommand{\expan}[1]{\lt\langle #1\rt\rangle}           						
\newcommand{\sexpan}[1]{\llangle #1\rrangle}         									
\newcommand{\floor}[1]{\lt\lfloor #1\rt\rfloor}												
\newcommand{\F}{\ensuremath{\mathbb{F}}}															
\newcommand{\A}{\ensuremath{\mathcal{A}}}															
\newcommand{\m}{\ensuremath{\mathbf{m}}}															
\DeclareMathOperator{\supp}{supp} 																		
\DeclareMathOperator{\anc}{anc}  																			
\DeclareMathOperator{\pdb}{pd_{\it b}}																
\DeclareMathOperator{\db}{d_{\it b}}																	

\makeatletter
	\def\imod#1{\allowbreak\mkern10mu({\operator@font mod}\,\,#1)} 
	\def\wmod#1{\allowbreak\mkern10mu{\operator@font mod}\,\,#1} 
	\def\@setcopyright{}                                           
	\def\serieslogo@{}
\makeatother

\begin{document}
	\author[M.J.C.~Loquias]{Manuel Joseph C.~Loquias}
	\address{Chair of Mathematics and Statistics, University of Leoben, Franz-Josef-Strasse 18, A-8700 Leoben,
	Austria, and Institute of Mathematics, University of the Philippines Diliman, 1101 Quezon City, Philippines}
	\email{mjcloquias@math.upd.edu.ph}

	\author[M.~Mkaouar]{Mohamed Mkaouar}
	\address{Facult\'e des Sciences de Sfax, BP 802, Sfax 3018, Tunisia} 
	\email{mohamed.mkaouar@fss.rnu.tn}

	\author[K.~Scheicher]{Klaus Scheicher}
	\address{Institute for Mathematics, University of Natural Resources and Applied Life Sciences, Gregor Mendel Strasse 33, A-1180 Wien, AUSTRIA}
	\email{klaus.scheicher@boku.ac.at}

	\author[J.M.~Thuswaldner]{J\"org M.~Thuswaldner}
	\address{Chair of Mathematics and Statistics, University of Leoben, Franz-Josef-Strasse 18, A-8700 Leoben, Austria}
	\email{joerg.thuswaldner@unileoben.ac.at}

	\title{Rational digit systems over finite fields and Christol's Theorem}

		\begin{abstract} 
		Let $P, Q\in \F_q[X]\setminus\{0\}$ be two coprime polynomials over the finite field $\F_q$ with $\deg P > \deg Q$. We represent each polynomial $w$ over $\F_q$ by 
		\[
		w=\sum_{i=0}^k\frac{s_i}{Q}{\lt(\frac{P}{Q}\rt)}^i
		\] 	
		using a rational \emph{base} $P/Q$ and \emph{digits} $s_i\in\F_q[X]$ satisfying $\deg s_i < \deg P$. 
		\emph{Digit expansions} of this type are also defined for formal Laurent series over $\F_q$.
		We prove uniqueness and automatic properties of these expansions. 
		Although the $\omega$-language of the possible digit strings is not regular, we are able to characterize the digit expansions of algebraic elements. 
		In particular, we give a version of Christol's Theorem by showing that the digit string of the digit expansion of a formal Laurent series is automatic 
		if and only if the series is algebraic over $\F_q[X]$. 
		Finally, we study relations between digit expansions of formal Laurent series and a finite fields version of Mahler's $3/2$-problem. 
	\end{abstract}



	\date{\today}

	\thanks{The first and fourth authors were supported by project I1136 ``Fractals and numeration'' granted by the Austrian Science Fund (FWF).  
	The first author was also supported by an Austrian Exchange Service (OeAD-GmbH) scholarship financed by the Austrian Federal Ministry of Science, 
	Research and Economy (BMWFW) within the framework of the ASEA UNINET. The third author was supported by project P23990 ``Number systems for Laurent series over finite fields'' granted by the Austrian Science Fund (FWF)}

\dedicatory{Dedicated to Professor Peter Grabner on the occasion of his $50$th birthday}

	\maketitle
	
	\section{Introduction}
		We study digit systems with ``rational bases'' defined in rings of polynomials and fields of formal Laurent series over finite fields. 
		Although the $\omega$-language of the digit strings of expansions with respect to such a digit system is not regular, expansions of algebraic elements are well behaved. 
		In particular, we are able to establish a version of Christol's Theorem for expansions of formal Laurent series. 
		
		Digit systems over finite fields have been first studied in 1991 by Kov\'acs and Peth\H{o}~\cite{KP91}. 
		Since then they have been generalized in various ways (see~\cite{BBST:09,R:08,S07,ST03}) 
		and many of their properties have been investigated ({\it cf.}~\cite{AH:13,AHPPS,BGM:14,LP:08,LWX:08,MW:15,SSTV14}). 
		As indicated in Effinger {\it et al.}~\cite{EHM} theories sometimes reveal new features when transferred from the integer to the finite fields setting. 
		This is true also for digit systems over finite fields: although they often behave similar to their ``cousins'' defined over $\mathbb{Z}$ and $\mathbb{R}$, 
		they also show completely different properties that have no analogue in the integer case.  
		This is highlighted for instance by the $p$-automaticity results by Rigo~\cite{R:08}, Rigo and Waxweiler~\cite{RW11} and, more recently, by Scheicher and Sirvent~\cite{SS}. 
		The specialty of the digit systems studied in the present paper consists in the fact that they have rational functions as bases and, 
		hence, form the analogues of the rational based digit systems studied by Akiyama {\it et al.}~\cite{AFS}.

		The rational based digit systems studied by Akiyama {\it et al.}~\cite{AFS} have a very difficult structure and are related to the notorious $3/2$-problem of Mahler~\cite{M68}. 
		The present contribution aims at getting more information on the structure of their finite fields analogues. 
		Although some properties that we face are similar to the ones in the integer setting, 
		we put emphasis in proving results that are special to the finite fields setting.  
		These new results are often related to automaticity properties of the digit systems over finite fields. 
		As mentioned above, one of our main results is a version of Christol's Theorem ({\it cf.}~Theorem~\ref{Christol}). 
		To prove this we derive explicit formulas for the digit expansion of a given formal Laurent series (see Theorem~\ref{digitLaurent}). 
		A motivation of these results comes from the fact that they contrast the results by Adamczewski and Bugeaud~\cite{AB07} 
		stating that in the case of $q$-ary number systems, automaticity of expansions gives rise to transcendental numbers. 
		Indeed, in the case of our Laurent series expansions, automaticity of expansions yields algebraicity of the expanded element.
		
		To be more precise, let $P$ and $Q$ be nonzero and coprime polynomials over a finite field $\F_q$ whose degrees satisfy $\deg P > \deg Q$. 
		Then by a simple algorithm we can expand each polynomial $w$ over $\F_q$ by its \emph{$P/Q$-polynomial digit expansion} 
		\[
			w=\sum_{i=0}^k\frac{s_i}{Q}{\lt(\frac{P}{Q}\rt)}^i
		\] 	
		with ``digits'' $s_i$ that are polynomials over $\F_q$ satisfying $\deg s_i < \deg P$ (see Section~\ref{sec:2}). 
		While in the integer case the least significant digits are periodic, 
		they form an automatic sequence and can even be generated using substitutions in the present setting (Section~\ref{sec:3}).
		In addition, Corollary~\ref{multbyR} tells us that the multiplication of a polynomial by a fixed polynomial can be realized by a finite transducer automaton. 
		
		The expansion graph characterizing the possible $P/Q$-polynomial digit expansions is set up in Section~\ref{sec:4} in a similar way as in the integer setting. 
		However, in the finite fields case this graph is, up to a self-loop at $0$, a complete $(q^{\deg P-\deg Q})$-ary tree. 
		Nevertheless, due to the complicated structure of its labels, the graph contains no nontrivial infinite string that is eventually periodic (see Theorem~\ref{eventperiod}). 
		The set of infinite strings in the graph may be viewed as a language called an \emph{$\omega$-language}.
		Analogous to the finite case, an $\omega$-language is said to be \emph{regular} if it is accepted by a {\it B\"uchi automaton}
		(for details see \emph{e.g.}~\cite{L97,W90}).  
		We see in Theorem~\ref{omegalang} that the $\omega$-language of all infinite strings in the expansion graph is not regular. 
		
		In Akiyama~{\it et al.}~\cite{AFS}, not only integers are represented in terms of positive powers of rational numbers. 
		The authors extend their notion of rational based number systems to negative powers in order to represent real numbers. 
		Like $\F_q[X]$ being the analogue of $\Z$ in the finite fields setting, the field of Laurent series $\F_q((X^{-1}))$ furnishes the analogue of the real numbers. 
		Indeed, it turns out that $P/Q$-digit expansions can also be defined for formal Laurent series over $\F_q$ by means of the infinite strings in the expansion graph 
		(which is detailed in Section~\ref{sec:5}) and that they enjoy unicity properties (Theorem~\ref{wpq}). 
		We mention that an extension in the positive direction to the ring of formal power series $\F_q[[X]]$ is problematic 
		in view of issues with the determination of a natural unique residue class for the digits (see Remark~\ref{rem:powerseries} for details on this).
 		
		Section~\ref{sec:6a} contains two theorems that allow to compute the $P/Q$-digit expansions of polynomials and of formal Laurent series by closed formulas. 
		The latter formula is then used in Section~\ref{sec:7} to derive a version of Christol's Theorem: we are able to show that the algebraic elements in the field of formal Laurent 
		series over $\F_q$ are characterized by automatic $P/Q$-digit expansions. 
		As mentioned above this is surprising as the structure of the expansion graph is complicated. 
		Finally, in Section~\ref{sec:6} we relate $P/Q$-digit expansions of formal Laurent series to a finite fields version of Mahler's $3/2$-problem.			
	
	\section{Expansions of polynomials with respect to rational bases }\label{sec:2}
		Denote by $\F_q$ the finite field of $q$ elements and let $\F_q[X]$, $\F_q(X)$, and $\F_q((X^{-1}))$ be the ring of polynomials over $\F_q$, 
		the field of rational functions over $\F_q$, and the field of formal Laurent series over $\F_q$, respectively. 
		In this section we define digit systems for $\F_q[X]$ whose bases are rational functions. 
		These digit systems provide a finite fields analogue of the digit systems studied in Akiyama {\it et al.}~\cite{AFS}. 	        
	      
	  \begin{definition}[{$P/Q$-digit systems in $\F_q[X]$}]
			Let $P,Q\in\F_q[X]\setminus\{0\}$, with $\deg{P}>\deg{Q}$, be two coprime polynomials and let
			$\mathcal{D}\vcentcolon=\set{s\in\F_q[X]\,:\,\deg{s}<\deg{P}}$.
			Then $(P/Q,\mathcal{D})$ is called a \emph{rational function based digit system} with \emph{base} $P/Q$ and 
			\emph{set of digits} $\mathcal{D}$ (\emph{$P/Q$-digit system} for short). 
			An expansion of an element $w\in \F_q[X]$ of the form
			\begin{equation}\label{e4}
				w=\sum_{i=0}^k\frac{s_i}{Q}{\lt(\frac{P}{Q}\rt)}^i,
			\end{equation}		
			 with $s_i\in \mathcal{D}$ is called \emph{$P/Q$-polynomial digit expansion} of~$w$.
		\end{definition}
	
		\begin{remark}
			If $\deg Q = 0$, {\it i.e.}, $Q\in \F_q\setminus\{0\}$, the $(P/Q)$-digit system is a special case of the digit systems studied in \cite{ST03}. 
			In this case the language of the digit strings is the full shift.
		\end{remark}
						
	 	Let $(P/Q,\mathcal{D})$ be a rational function based digit system. 
		To see that each $w\in \F_q[X]$ admits a $P/Q$-polynomial digit expansion we introduce the following algorithm. 
		Set $w_0 = w$. For $i\geq 0$, define $w_{i+1}\in \F_q[X]$ and $s_i\in \mathcal{D}$ recursively by the equation
		\begin{equation}\label{e3}
			Qw_i=Pw_{i+1}+s_i,
		\end{equation}
		where $s_i$ is the unique remainder of the division of $Qw_i$ by $P$. 
		Since  $\deg w_{i+1}<\deg w_{i}$ holds for $w_i\not=0$, there exists a minimal $k$ such that $w_{i}=0$ for each $i>k$. 
		Applying \eqref{e3} for $k+1$ times produces the \emph{digits} $s_0,\ldots, s_k$ and yields the expansion \eqref{e4} for $w$.
			
		In order to show that the $P/Q$-polynomial digit expansion generated by this algorithm is the only $P/Q$-polynomial digit expansion for $w\in\F_q[X]$, we need the following lemma.

		\begin{lemma}\label{th1}
			Let $(P/Q,\mathcal{D})$ be a rational function based digit system. If 
			\begin{equation}\label{eq:unique}
			\sum_{i=0}^k s_i{\lt(\frac{P}{Q}\rt)}^i=0,
			\end{equation}
			where
			$s_i\in\mathcal{D}$ for each $i\in\set{0,\ldots,k}$, 
			then $s_i=0$ for each $i\in\set{0,\ldots,k}$.
		\end{lemma}
		\begin{proof}
			Suppose that this is wrong. Then there is a minimal $d\in \mathbb{N}$ with $s_d\not=0$. 
			Multiplying \eqref{eq:unique} by $Q^kP^{-d}$ yields $s_dQ^{k-d} = PR$ for some $R\in \F_q[X]$. 
			Since $P$ and $Q$ are coprime, $P$ divides $s_d$. As  $\deg s_d < \deg P$, this implies that $s_d = 0$, a contradiction.
		\end{proof}

		Consequently, we obtain the following result.

		\begin{theorem}\label{integerRepresentation}
			Let $(P/Q,\mathcal{D})$ be a rational function based digit system. Then each $w\in\F_q[X]$ admits a unique $P/Q$-polynomial digit expansion.
		\end{theorem}

		The \emph{digit string} associated with a $P/Q$-polynomial digit expansion \eqref{e4} of a polynomial $w\in \F_q[X]$ 
		(which is unique in view of Theorem~\ref{integerRepresentation}) will be denoted by 
		\[\expan{w}_{P/Q}=(s_k\cdots s_0)_{P/Q}.\] 
		Note that $\expan{w}_{P/Q}$ is always a finite string over the alphabet $\mathcal{D}$.
		For the sake of simplicity of notation, we will sometimes just write $s_k\cdots s_0$ instead of $(s_k\cdots s_0)_{P/Q}$. 
		
		We denote the set of finite strings over $\mathcal{D}$ by
		\[\mathcal{D}^{\ast}=\set{s_k\cdots s_0\,:\, s_i\in\mathcal{D},k\ge 0} \cup \{\varepsilon\},\]
		where $\varepsilon$ denotes the empty string. To each string $s_0\cdots s_k\in \mathcal{D}^{\ast}$ we can associate an element of $\F_q(X)$ by the \emph{evaluation map}
		\begin{equation}\label{evalmap}
			\pi:\mathcal{D}^{\ast}\rightarrow\F_q(X)\;,\;s_k\cdots s_0  \mapsto\sum_{i=0}^k\frac{s_i}{Q}{\lt(\frac{P}{Q}\rt)}^i.
		\end{equation}
		The following simple example illustrates that $\pi(\mathcal{D}^{\ast})\subseteq\F_q[X]$ is not true in general.
		
		\begin{example}\label{ex1}
			Let $(P/Q,\mathcal{D})$ be the $P/Q$-digit system in $\F_2[X]$ with $P=X^2+1$ and $Q=X$.  
			Then $\mathcal{D}=\set{0,1,X,X+1}$. Table~\ref{tab1} shows some $P/Q$-polynomial digit expansions. 
			Observe that the string $1\in\mathcal{D}^{\ast}$ is mapped to $\pi(1)_{P/Q}=X^{-1}\not\in\F_q[X]$.

			\begin{table}[h]\label{tab1} 
				\begin{tabular}{|r|r|}
				\hline
				$w\in\F_2[X]$ & $\expan{w}_{P/Q}$\\\hline\hline
				$0$ & $(0)_{P/Q}$\\\hline
				$1$ & $(X)_{P/Q}$\\\hline
				$X$ & $(X,1)_{P/Q}$\\\hline
				$X+1$ & $(X,X+1)_{P/Q}$\\\hline
				$X^2$ & $(X,1,X)_{P/Q}$\\\hline
				$X^2+1$ & $(X,1,0)_{P/Q}$\\\hline
				$X^2+X$& $(X,X+1,X+1)_{P/Q}$\\\hline
				$X^2+X+1$& $(X,X+1,1)_{P/Q}$\\\hline
				\end{tabular}
				\bigskip
				\caption{Polynomial digit expansions of some elements of $\F_2[X]$ with respect to the rational function based digit system with base $(X^2+1)/X$. 
				Digits are separated by commas for the sake of readability.}
			\end{table}			
		\end{example}	
				
	\section{Automatic properties of $P/Q$-polynomial digit expansions}\label{sec:3}
		Let $(P/Q,\mathcal{D})$ be a rational function based digit system. 
		In this section, we look at properties of the language $\mathcal{L}_{P/Q}$ formed by the digit strings of $P/Q$-polynomial digit expansions with special emphasis on automaticity. 
		We first recall some definitions.
		
		\begin{definition}[{Transducer and deterministic finite automaton; {\it cf.}~\cite[Chapter~4]{AS}}]\label{transducer}
			The 6-tuple $\mathcal{A}=(\mathcal{Q},\Sigma,q,u_0,\Delta,\delta)$ is called a \emph{letter-to-letter finite state transducer}
			if $\mathcal{Q}$, $\Sigma$, and $\Delta$ are nonempty, finite sets, $u_0\in\mathcal{Q}$, and
			$q:\mathcal{Q}\times\Sigma\rightarrow \mathcal{Q}$ and $\delta:\mathcal{Q}\times\Sigma\rightarrow\Delta$ are mappings.
			The sets $\Sigma$ and $\Delta$ are called \emph{input} and \emph{output alphabet}, respectively. 
			We denote by $\Sigma^{\ast}$ and $\Delta^{\ast}$ the sets of finite strings of elements of $\Sigma$ and $\Delta$, respectively.
			The set $\mathcal{Q}$ is called the \emph{set of states}, and the transducer $\mathcal{A}$ starts at the \emph{initial state} $u_0$.
			The mappings $q$ and $\delta$ are called \emph{transition} and \emph{result function}, respectively.  
			Throughout this paper, we adopt the convention that transducers and automata read
			input strings from right to left.  If the input to the transducer $\mathcal{A}$ is the string $s_k\cdots s_0\in\Sigma^\ast$, then 
			the output string $d_k\cdots d_0\in\Delta^\ast$ is computed recursively by
			\begin{align*}
				d_i&=\delta(u_i,s_i),&u_{i+1}&=q(u_i,s_i),\qquad (0\le i < k);\\
				d_k&=\delta(u_k,s_k).
			\end{align*}
			If $\delta$ is instead the function $\delta:\mathcal{Q}\rightarrow\Delta$, then $\mathcal{A}$ is
			called a \emph{deterministic finite automaton with output} (DFAO for short) ({\it cf.}~\cite[p.~138]{AS}).  
			In this case the output is just the element $\delta(u_k)$.
			
			If the DFAO has no output function (and, hence, no output alphabet) it is called a \emph{deterministic finite automaton} (DFA for short).
			
			Let $\mathcal{A}$ be a DFA. 
			Extend the definition of the transition function $q$ to $\mathcal{Q}\times \Sigma^{\ast}$ by defining $q(u,\varepsilon)=u$ for $u\in \mathcal{Q}$ 
			and $\varepsilon$ the empty string, and $q(u,as)=q(q(u,s),a)$ for all $u \in \mathcal{Q}$, $s\in \Sigma^{\ast}$, and $a\in \Sigma$. 
			For each $F\subseteq\mathcal{Q}$ the language 
			\[\mathcal{L}_F=\mathcal{L}_F(\mathcal{A})=\{s\in\Sigma^\ast\,:\,q(u_0,s)\in F\}\]
			is called a \emph{language accepted by $\mathcal{A}$}.
			A language $\mathcal{L}$ is said to be \emph{regular} if it is accepted by a DFA.  
		
			If $a,b,c\in\mathcal{L}$ such that $a=bc$, then we say that $b$ is a \emph{prefix} and that $c$ is a \emph{suffix} of $a$. 
			We call $\mathcal{L}$ \emph{prefix--closed} if it contains every prefix of its elements.  
			A \emph{suffix--closed} language is defined analogously.
		\end{definition}

		We note that a letter-to-letter finite state transducer (as well as a DFAO or DFA) can be represented by a graph. 
		Indeed the states of the graph are the states of the transducer. 
		Moreover, there is a labelled edge  $u_1\xrightarrow{(a,b)} u_2$ in this graph  if and only if $q(u_1,a)=u_2$ and $\delta(u_1,a)=b$. 
		See Figure~\ref{transex} for an example of a graph representation of a transducer.
		
		\begin{theorem}\label{lang} 
			Let $(P/Q,\mathcal{D})$ be a rational function based digit system. 
			If $\deg Q \ge 1$ then the language $\mathcal{L}_{P/Q}=\{\expan{w}_{P/Q}\,:\, w\in\F_q[X]\}$ is not regular and not suffix--closed. 
			Nevertheless, $\mathcal{L}_{P/Q}$ is prefix--closed.
		\end{theorem}
		\begin{proof}
			A look at Table~\ref{tab1} shows that $\mathcal{L}_{P/Q}$ is not suffix-closed. 
			The fact that $\mathcal{L}_{P/Q}$ is prefix--closed follows from the expansion algorithm in~\eqref{e3}. 
			It remains to prove nonregularity.
			
			Denote the number of letters of a word $s$ by $|s|$ and define for $w\in \F_q[X]$ and $k\in \mathbb{N}$ the set
			\[R_k(w)=\{s\in\mathcal{D}^\ast\,:\,|s|=k\text{ and }\langle w\rangle_{P/Q}s\in\mathcal{L}_{P/Q}\},\]
			which is nonempty by the definition of $\mathcal{L}_{P/Q}$.
			Let $v,w\in \F_q[X]$. 
			We claim that a word $s\in \mathcal{D}^\ast$ with $|s|=k$ can belong to $R_k(v)\cap R_k(w)$ only if $v\equiv w\pmod{Q^k}$. 
			Indeed if $s\in R_k(v)\cap R_k(w)$ then $(P/Q)^kv+\pi(s),(P/Q)^kw+\pi(s)\in\F_q[X]$. 
			Taking differences and multiplying by $Q^k$, the fact that $P$ and $Q$ are coprime yields that $v\equiv w\pmod{Q^k}$. 
			
			On the language $\mathcal{L}_{P/Q}$ the \emph{Myhill-Nerode equivalence relation $\mathcal{R}$} is defined by $\langle v\rangle_{P/Q}\mathcal{R}\langle w \rangle_{P/Q}$ 
			if and only if for all $s\in\mathcal{D}^\ast$ one has $\langle v\rangle_{P/Q}s\in \mathcal{L}_{P/Q}$ if and only if $\langle w\rangle_{P/Q}s\in\mathcal{L}_{P/Q}$. 
			It follows from the previous paragraph that $\mathcal{R}$ has infinite index. 
			Thus the  classical Theorem of Myhill--Nerode (see~{\it e.g.}~\cite[Theorem~4.1.8]{AS}) implies that $\mathcal{L}_{P/Q}$ is not regular. 		
		\end{proof}
				
		Even if the language $\mathcal{L}_{P/Q}$ is not regular, we see in the subsequent results that it still has several automatic properties. 
		First recall the following definition.
		
		\begin{definition}[{$\F_q$-automatic function; {\it cf.}~\cite[p.~138]{AS}}]
			Let $\Delta$ be a nonempty finite set. 
			We say that a function $f:\F_q[X]\rightarrow\Delta$ is an \emph{$\F_q$-automatic finite state function} (or \emph{$\F_q$-automatic} for short) 
			if $f(w)$ can be realized as the result function of a DFAO with input alphabet $\F_q$ and whose input is the string $w_k\cdots w_1w_0$ 
			that is formed from the coefficients of $w=w_0+w_1X+\cdots+w_kX^k\in\F_q[X]$.
		\end{definition}

		We shall need the following auxiliary lemma.  Recall that there is a canonical bijection between $\mathcal{D}=\set{s\in\F_q[X]\,:\,\deg{s}<\deg{P}}$ and $\F_q[X]/P$ 
		that assigns to each digit in $\mathcal{D}$ its residue class modulo $P$.

		\begin{lemma}\label{l4}
			Let $P,Q\in\F_q[X]$ with $P\ne 0$. Then the following assertions hold:
			\begin{enumerate}
				\item The function $f:\F_q[X]\to\mathcal{D}$, $f(w)=Qw\imod{P}$ is $\F_q$-automatic.
				
				\item If $f_1,f_2:\F_q[X]\to\mathcal{D}$ are $\F_q$-automatic functions and $c\in\F_q(X)$, then $f_1\pm f_2$ and $cf_1$ are
				$\F_q$-automatic.
			\end{enumerate}
		\end{lemma}
		\begin{proof}
			  To prove (i) let $w=w_0+w_1X+\cdots+w_kX^k\in\F_q[X]$.
				A standard result of automata theory yields that a function is $\F_q$-automatic in direct reading if and only if 
				it is $\F_q$-automatic in reverse reading ({\it cf.}~\cite[Theorem~4.3.3]{AS}). 
				Let ${\mathcal A}=(\mathcal{Q},\Sigma,q,u_0,\Delta,\delta)$ be the DFAO defined as follows: $\mathcal{Q}=\Delta=\mathcal{D}$, $\Sigma=\F_q$, $u_0=0$,	
				\[
				\begin{array}{rrclrcl}
					q:&\mathcal{Q}\times\Sigma&\mapsto&\mathcal{Q}, &q(A,a)&=&AX+a\imod{P},\\			
					\delta:&\mathcal{Q}&\mapsto&\Delta, &\delta(A)&=&QA\imod{P}.
				\end{array}
				\]
				Then direct calculation shows that $f(w)=\delta(u_{k+1})$ if $\mathcal{A}$ is feeded with the (reverse) input string $w_0w_1\cdots w_k$.
				
				Since (ii) is trivial the proof is finished.
		\end{proof}

		For a given $P/Q$-digit system let $s^{(m)}:\F_q[X]\to\mathcal{D}$ be the function that picks the digit $s_m$ from the $P/Q$-polynomial digit expansion \eqref{e4}, 
		{\it i.e.}, for $w\in\F_q[X]$ with $\expan{w}_{P/Q}=(s_k\cdots s_0)_{P/Q}$ we set $s^{(m)}(w)=s_m$ for $m\le k$ and $s^{(m)}(w)=0$ otherwise.
		
		\begin{theorem}\label{th2}
			Let $(P/Q,\mathcal{D})$ be a rational function based digit system. Then, for each $m\geq 0$, the function $s^{(m)}$ is $\F_q$-automatic.
		\end{theorem}
		\begin{proof}
			The functions $s^{(m)}$ are defined recursively for each $w$.
 			Indeed, we let \mbox{$w^{(0)}=w$} and by~\eqref{e3},
			\begin{equation}\label{e5}
				Qw^{(m)}=Pw^{(m+1)}+s^{(m)}(w),\text{ with }\deg{s^{(m)}(w)}<\deg{P}.
			\end{equation}
			Induction yields
			\begin{equation}\label{e6}
				Q^{m+1}w^{(0)}=P^{m+1} w^{(m+1)}+\sum_{j=0}^{m}P^jQ^{m-j}s^{(j)}(w).
			\end{equation}
			Now, we prove by induction on $m$ that $s^{(m)}$ is $\F_q$-automatic.

			In the case $m=0$ it follows from \eqref{e5} that $s^{(0)}(w)=Qw\imod{P}$. 
			Hence, $s^{(0)}$ is $\F_q$-automatic by Lemma~\ref{l4}~(i).

			Suppose that $m>0$ and that $s^{(j)}$ is $\F_q$-automatic for $j<m$. 
			From Lemma~\ref{l4}~(i), it follows that the function $f(w)=Q^{m+1}w\imod{P^{m+1}}$ is $\F_q$-automatic. 
			By the induction hypothesis and Lemma~\ref{l4}~(ii), each function $P^jQ^{m-j}s^{(j)}(w)$ is $\F_q$-automatic for $0\leq j<m$.
			By~\eqref{e6}, we have
			\[P^ms^{(m)}(w)=Q^{m+1}w^{(0)}-\sum_{j=0}^{m-1}P^jQ^{m-j}s^{(j)}(w)\imod{P^{m+1}},\]
			which is $\F_q$ automatic by Lemma~\ref{l4}~(ii).
			Thus, we see by applying Lemma~\ref{l4}~(ii) again that $s^{(m)}(w)=P^{-m}P^ms^{(m)}(w)$ is also $\F_q$-automatic.
		\end{proof}		
		
		Consider an order $\sqsubseteq$ of the finite field $\F_q$ that starts with $0$ and~$1$.  
		It induces a lexicographic order $\sqsubseteq$ on $\F_q[X]$ defined as follows: 
		Given polynomials $v=v_0+v_1X+\cdots+v_kX^k$ and $w=w_0+w_1X+\cdots+w_{\ell}X^{\ell}$ in $\F_q[X]$,
		then $v\sqsubset w$ if $k<\ell$, or if $k=\ell$ and there exists $i$ such that $v_i\sqsubset w_i$ with $v_j=w_j$ for all $i<j\leq k$.
		
		A substitution rule $\varrho:\Sigma\to\Sigma^{\ast}$ on an alphabet $\Sigma$ is said to be \emph{$k$-uniform} 
		if it maps every element of $\Sigma$ to a string of $\Sigma^{\ast}$ of length $k$. 
		Denote by $a_n$ the $(n+1)$-st element of $\F_q[X]$ arranged according to the lexicographic order $\sqsubseteq$ on $\F_q[X]$. 
		With this notation we have $a_{qm+r}=Xa_m+a_r$ for $0\le r < q$,  $\F_q=\{a_0,\ldots, a_{q-1}\}$, and $\mathcal{D}=\{a_0,\ldots, a_{q^{\deg P }-1}\}$.
		Define the alphabet by $\Sigma=\mathcal{D}$ (which is again identified with the set of residue classes modulo $P$). 
		Then, using the abbreviation $R_i=Xa_i\imod{P}$, we define the $q$-uniform substitution $\varrho:\mathcal{D}\to\mathcal{D}^{\ast}$ by 
		(letters are separated by commas) 
		\begin{equation}\label{eq:rho}
			\varrho(a_i)= R_i+a_0Q,R_i+a_1Q,\ldots,R_i+a_{q-1}Q\qquad (0\leq i<q^{\deg{P}}).
		\end{equation}
		The next theorem states that the sequence $(s^{(0)}(a_n))$ can be obtained from this $q$-uniform substitution rule $\varrho$.
			
		\begin{theorem}\label{firstdigit}
			Let $(P/Q,\mathcal{D})$ be a rational function based digit system and 
			let $\varrho_0\varrho_1\cdots$ be the fixed point of the substitution $\varrho$ defined in~\eqref{eq:rho} that starts with $0$. 
			Then the sequence $(s^{(0)}(a_n))_{n\ge 0}$ of first digits of the $P/Q$-polynomial digit expansions of elements of $\F_q[X]$ 
			that are ordered by the lexicographic order $\sqsubseteq$ is equal to $\varrho_0\varrho_1\cdots$.
		\end{theorem}
		\begin{proof}
			We proceed by induction.
			Clearly, $\varrho_0=0=s^{(0)}(a_0)$.
			Suppose that $s^{(0)}(a_j)=\varrho_j$ for $j<n$.  As noted above, writing $n=qm+r$ with $r\in\set{0,\ldots,q-1}$, we have $a_n=Xa_m+a_r$.			
			As~\eqref{e3} implies $s^{(0)}(a_r)=a_rQ$ and $s^{(0)}(Xa_m)=Xs^{(0)}(a_m)\imod{P}$ we obtain
			\[s^{(0)}(a_n)=Xs^{(0)}(a_m)\imod{P}+a_rQ=X\varrho_m\imod{P}+a_rQ=\varrho_n,\]		
			where the last equality follows from the definition of $\varrho$.	
		\end{proof}
		
		\begin{remark}\label{cobhams}
			Since the sequence $(s^{(0)}(a_n))$ is a fixed point of a $q$-uniform substitution, 
			it follows from Cobham's Theorem ({\it cf.}~\cite[Theorem~6.3.2]{AS}) that it is $q$-automatic.
			That is, there exists a DFAO that gives $s^{(0)}(a_n)$ as output when the input is the string $w_0\cdots w_k$, where
			\[n=\sum_{i=0}^k{w_iq^i}\]
			with $w_i\in\set{0,\ldots,q-1}$ for $0\leq i\leq k$.
		\end{remark}

		\begin{example}\label{firstdigitex}
			If $P=X^2+1$ and $Q=X$ in $\F_2[X]$, then $a_0=0$, $a_1=1$, $a_2=X$, and $a_3=X+1$.
			According to Theorem~\ref{firstdigit} the sequence $(s^{(0)}(a_n))$ is the fixed point of the $2$-uniform substitution $\varrho$ starting at $a_0$, where
			\[
				\varrho(a_0)=a_0a_2\;,\;
				\varrho(a_1)=a_2a_0\;,\;
				\varrho(a_2)=a_1a_3\;,\;
				\varrho(a_3)=a_3a_1.
			\]
			Indeed,  
			\[
				(s^{(0)}(a_n))=(a_0a_2a_1a_3a_2a_0a_3a_1a_1a_3a_0a_2a_3a_1a_2a_0a_2a_0a_3a_1a_0a_2a_1\cdots).
			\]
		\end{example}
		
		Our next aim is to obtain an analogue of the so-called odometer for rational function based digit systems. 
		This involves describing how to determine the expansion of $n+1$ from the expansion of $n$. 
		However, this is trivial for $P/Q$-digit systems because no ``carry'' occurs.
		Indeed, if $w\in\F_q[X]$ then the $P/Q$-polynomial digit expansion of $w+1$ is obtained by adding
		$\expan{1}_{P/Q}=Q$ to the least significant digit of $\expan{w}_{P/Q}$.
		It is more interesting to look instead at how to perform multiplication by $X$. 		
		
		In the remaining part of this section we view elements $s\in\mathcal{D}^{\ast}$ as $\varepsilon s$, where $\varepsilon$ is the empty string.  
		That is, if an input string $s=s_k\cdots s_0$ to the letter-to-letter finite state transducer $\mathcal{A}=(\mathcal{Q},\mathcal{D},q,u_0,\mathcal{D},\delta)$ 
		ends at the final state $u_k\in\mathcal{Q}$ and yields the output string $d_k\cdots d_0\in\mathcal{D}^{\ast}$, 
		then $\delta$ has to be defined at $(u_k,\varepsilon)$ and the final output string of $\mathcal{A}$ is 
		$\mathcal{A}(s_k\cdots s_0)\vcentcolon=\delta(u_k,\varepsilon)d_k\cdots d_0$ ({\it cf.}~\cite{AFS}).

		\begin{theorem}\label{th3}
			Let $(P/Q,\mathcal{D})$ be a rational function based digit system.
			Then the function $\mathcal{L}_{P/Q}\rightarrow\mathcal{L}_{P/Q}$, $\expan{w}_{P/Q}\mapsto\expan{Xw}_{P/Q}$ 
			is realizable by a letter-to-letter finite state transducer.
		\end{theorem}
		
		\begin{proof}
			Let $\mathcal{A}=(\mathcal{Q},\mathcal{D},q,u_0,\mathcal{D},\delta)$ be the letter-to-letter finite state transducer, 
			where $\mathcal{Q}=\F_q$ and $u_0=0$.
			The transition function $q$ is defined as follows: if 
			\[P=p_0+\cdots+p_m X^m\quad\text{and}\quad t=t_{0}+\dots+t_{m-1}X^{m-1}\in\mathcal{D},\]
			then $u'\vcentcolon=q(u,t)$ where
			\[
			u'=\frac{t_{m-1}}{p_m}.
			\]
			Observe that $u'$ depends solely on the digit $t$.
			Finally, the result function $\delta$ is given by $d\vcentcolon=\delta(u,t)$ where
			\[
			d=Xt-(t_{m-1}/p_m)P+uQ=Xt-u'P+uQ,
			\]
			and if $u\in\mathcal{Q}$, then $\delta(u,\varepsilon)\vcentcolon =uQ$.
			Note that $\delta$ is well defined because both \mbox{$\deg(Xt-u'P)$} and $\deg{Q}$ are less than $\deg{P}$.  
			
			The proof is done by induction on the length of $\expan{w}_{P/Q}$.
			For the induction start consider $\expan{w}_{P/Q}=s_0\in\mathcal{D}$. 
			If $u_1=q(0,s_0)$, then we obtain 
			\begin{align*}
				\mathcal{A}(s_0)     &=(u_1Q,Xs_0-u_1P), \text{ and}\\
				\pi(\mathcal{A}(s_0))&=\frac{Xs_0-u_1P}{Q}+\frac{u_1Q}{Q}\lt(\frac{P}{Q}\rt)=X\frac{s_0}{Q}=X\pi(s_0),
			\end{align*}
			where $\pi$ is the evaluation map defined in \eqref{evalmap}.
			
			Assume that the statement holds for any polynomial over $\F_q$ whose $P/Q$-digit expansion is of length $k$.
			Consider a polynomial $w\in\F_q[X]$ of length $k+1$ and write $\expan{w}_{P/Q}=s_k\cdots s_1s_0$.
			We now proceed to show that $\expan{Xw}_{P/Q}=\mathcal{A}(s_k\cdots s_1s_0)$, or equivalently, $\pi(\mathcal{A}(s_k\cdots s_1s_0))=Xw$.
			To this end, let $\mathcal{A}(s_k\cdots s_1s_0)=d_{k+1}\cdots d_1d_0$.  Recall that the initial state is $u_0=0$, 
			\begin{equation}\label{outputk+1}
				u_{i+1}\vcentcolon=q(u_i,s_i)\qquad\text{and}\qquad d_i=Xs_i-u_{i+1}P+u_iQ
			\end{equation}
			for $0\leq i\leq k$, and $d_{k+1}=u_{k+1}Q$.
			
			Now, since $\mathcal{L}_{P/Q}$ is prefix-closed by Theorem~\ref{lang}, we have $s_k\cdots s_1\in\mathcal{L}_{P/Q}$. 
			Hence, there exists $w'\in\F_q[X]$ such that $\expan{w'}_{P/Q}=s_k\cdots s_1$ and satisfying
			\begin{equation}\label{induction}
				w=\frac{P}{Q}w'+\frac{s_0}{Q}.
			\end{equation}
			Suppose $\mathcal{A}(s_k\cdots s_1)=d_{k+1}'\cdots d_1'$.  In this case, the initial state is $u_1'=0$,
			\begin{equation}\label{outputk}
				u_{i+1}'\vcentcolon=q(u_i',s_i)\qquad\text{and}\qquad d_i'=Xs_i-u_{i+1}'P+u_i'Q
			\end{equation}
			for $1\leq i\leq k$, and $d_{k+1}'=u_{k+1}'Q$.
			Since the transition function $q$ does not depend on the state, we obtain from~\eqref{outputk+1} and~\eqref{outputk} that $u_i'=u_i$ for $2\leq i\leq k$.  
			Thus, we also have $d_i'=d_i$ for $2\leq i\leq k$.
			We also see from ~\eqref{outputk+1} and~\eqref{outputk} that $d_0$ and $d_1$ are given by
			\begin{align*}
				d_0&=Xs_0-u_1P, \text{ and}\\
				d_1&=Xs_1-u_2P+u_1Q=(d_1'+u_2'P)-u_2P+u_1Q=d_1'+u_1Q.
			\end{align*}
			Since 
			\[\pi(\mathcal{A}(s_k\cdots s_1))=\sum_{i=1}^{k+1}\frac{d_i'}{Q}{\lt(\frac{P}{Q}\rt)}^{i-1}=Xw'\]
			holds by the induction hypothesis, we finally have
			\begin{align*}
				\pi(\mathcal{A}(s_k\cdots s_1s_0))&=\sum_{i=0}^{k+1}\frac{d_i}{Q}{\lt(\frac{P}{Q}\rt)}^i\\
				&=\frac{Xs_0-u_1P}{Q}+\frac{d_1'+u_1Q}{Q}\lt(\frac{P}{Q}\rt)+\sum_{i=2}^{k+1}\frac{d_i'}{Q}{\lt(\frac{P}{Q}\rt)}^i\\
				&=\frac{Xs_0}{Q}+\frac{P}{Q}\sum_{i=1}^{k+1}\frac{d_i'}{Q}{\lt(\frac{P}{Q}\rt)}^{i-1}\\
				&=\frac{Xs_0}{Q}+\frac{P}{Q}(Xw')\\
				&=\frac{Xs_0}{Q}+X\lt(w-\frac{s_0}{Q}\rt)=Xw,
			\end{align*}		
			where the last equality follows from~\eqref{induction}.
		\end{proof}

		\begin{example}
			Recall that if $P=X^2+1, Q=X\in\F_2[X]$, then $\mathcal{D}=\set{0,1,X,X+1}$.  
			Figure~\ref{transex} shows the representation of the transducer that realizes multiplication by $X$ in this $P/Q$-digit system.		
						
			\begin{figure}[ht]
				\includegraphics[height=3.5cm]{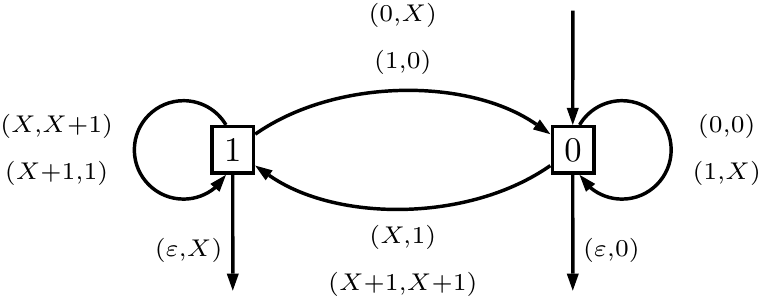}				
				\caption{The transducer for multiplication by $X$ in the $P/Q$-digit system with $P=X^2+1,Q=X\in\F_2[X]$ 
				(see Definition~\ref{transducer} and the remark after it for a definition of this graph; 
				the edges labelled by $(\varepsilon,X)$ and $(\varepsilon,0)$ are explained before Theorem~\ref{th3}). }
				\label{transex}
			\end{figure}
			
			For instance, since $\expan{X^2}_{P/Q}=(X,1,X)_{P/Q}$, we obtain from Figure~\ref{transex} that $\expan{X^3}_{P/Q}=(X,1,0,1)_{P/Q}$.
		\end{example}
		
		We note the following consequence of Theorem~\ref{th3}.
		
		\begin{corollary}\label{multbyR}
			Let $(P/Q,\mathcal{D})$ be a rational function based digit system and  $R\in\F_q[X]$.
			Then the function $\mathcal{L}_{P/Q}\rightarrow\mathcal{L}_{P/Q}$, $\expan{w}_{P/Q}\mapsto\expan{Rw}_{P/Q}$ is realizable by a letter-to-letter finite state transducer.
		\end{corollary}
		\begin{proof}
			The corollary is an immediate consequence of Theorem ~\ref{th3}, the composition theorem in \cite[Chapter~IV, Proposition~6.10]{Sakarovitch2009}, 
			and the fact that addition in the $P/Q$-digit system can be performed digitwise without carry by a letter-to-letter finite state transducer.
		\end{proof}

	\section{The expansion graph $T(P/Q)$}\label{sec:4}
		Let $(P/Q,\mathcal{D})$ be a rational function based digit system. 
		Since the language $\mathcal{L}_{P/Q}$ formed by the $P/Q$-polynomial digit expansions of elements of $\F_q[X]$ is prefix-closed by Theorem~\ref{lang}, 
		we will be able to define a graph $T(P/Q)$ containing the expansions of $(P/Q,\mathcal{D})$.
		Just as in the rational analogue considered in~\cite{AFS} we shall use this graph to construct unique $P/Q$-digit expansions of formal Laurent series over $\F_q$.
		
		\begin{definition}[Expansion graph]
			The \emph{expansion graph} $T(P/Q)=(V,E)$ is the edge-labelled infinite directed graph  
			satisfying $V=\F_q[X]$, and $(v,w)\in E$ with label $s\in\mathcal{D}$ whenever
			\[w=\frac{Pv+s}{Q}\in\F_q[X].\]
		\end{definition}

		Denote by $p(v)$ the string formed by the labels of the edges from the root $0$ to the node $v$ in $T(P/Q)$.
		It follows from~\eqref{e3} that $p(v)=\expan{v}_{P/Q}$. 

		The main difference of the graph $T(P/Q)$ from its rational analogue is that the number of outgoing edges for each node in $T(P/Q)$ is the same.

		\begin{theorem}\label{ktree}
			Let $(P/Q,\mathcal{D})$ be a rational function based digit system.
			The graph $T(P/Q)$ has a single loop which is a self-loop at $0$. Each node has $r$ outgoing edges, where $r=q^{\deg{P}-\deg{Q}}$.
		\end{theorem}
		\begin{proof}
			Consider a fixed $v\in\F_q[X]$.
			Note that $(v,w)\in E$ has label $s$ exactly when 
			\begin{equation}\label{minimal}
				s\equiv -Pv\imod{Q}. 
			\end{equation}
			In addition, $w$ is uniquely determined by $v$ and $s$. 
			Thus, the number $r$ of outgoing edges of $v$ is equal to the number of elements of $\set{s\in\mathcal{D}\,:\, s\equiv-Pv\imod{Q}}$.
			Hence, $r=q^{\deg{P}-\deg{Q}}$ as $\mathcal{D}$ contains all polynomials of degree less than $\deg{P}$. 
			The fact that $0$ is the only loop follows because for each edge $u\to v$ with $v\neq 0$ in $T(P/Q)$ we have $\deg{u}<\deg{v}$.
		\end{proof} 
		
		We say that a node $v$ of $T(P/Q)$ is of \emph{level $\ell$} if the shortest path between $0$ and $v$ is of length $\ell$. 
		By Theorem~\ref{ktree} there are $r^{\ell-1}$ nodes of level $\ell$ in $T(P/Q)$.
		In addition, these nodes are precisely the polynomials over $\F_q$ of degree $n$ where $(\ell-1)\cdot\deg(P/Q)\leq n<\ell\cdot\deg(P/Q)$.
		
		Let $L(v)$ be the set of the $r$ labels on outgoing edges from the node $v$.  
		It follows from~\eqref{minimal} that $L(v)$ contains exactly one digit of degree less than $\deg{Q}$.			
		If we denote this digit by $m_v$, then
		\begin{equation}\label{edgelabels}
			L(v)=\{m_v+dQ\,:\, d\in\F_q[X], \deg{d}<\deg(P/Q)\}.
		\end{equation}
		Thus, $L(v)$ is completely determined by $m_v$.  
		Hence, given $v,w\in\F_q[X]$, we have $L(v)=L(w)$ if and only if $v\equiv w\imod{Q}$.
		Consequently, there are exactly $q^{\deg{Q}}$ possible sets $L(v)$.
		
		\begin{example}\label{extree}
			If $P=X^2+1$ and $Q=X$ in $\F_2[X]$, then $r=2$.  The graph $T(P/Q)$ is shown in Figure~\ref{tree}.
			
			\begin{figure}[ht]	
			\includegraphics[height=13cm]{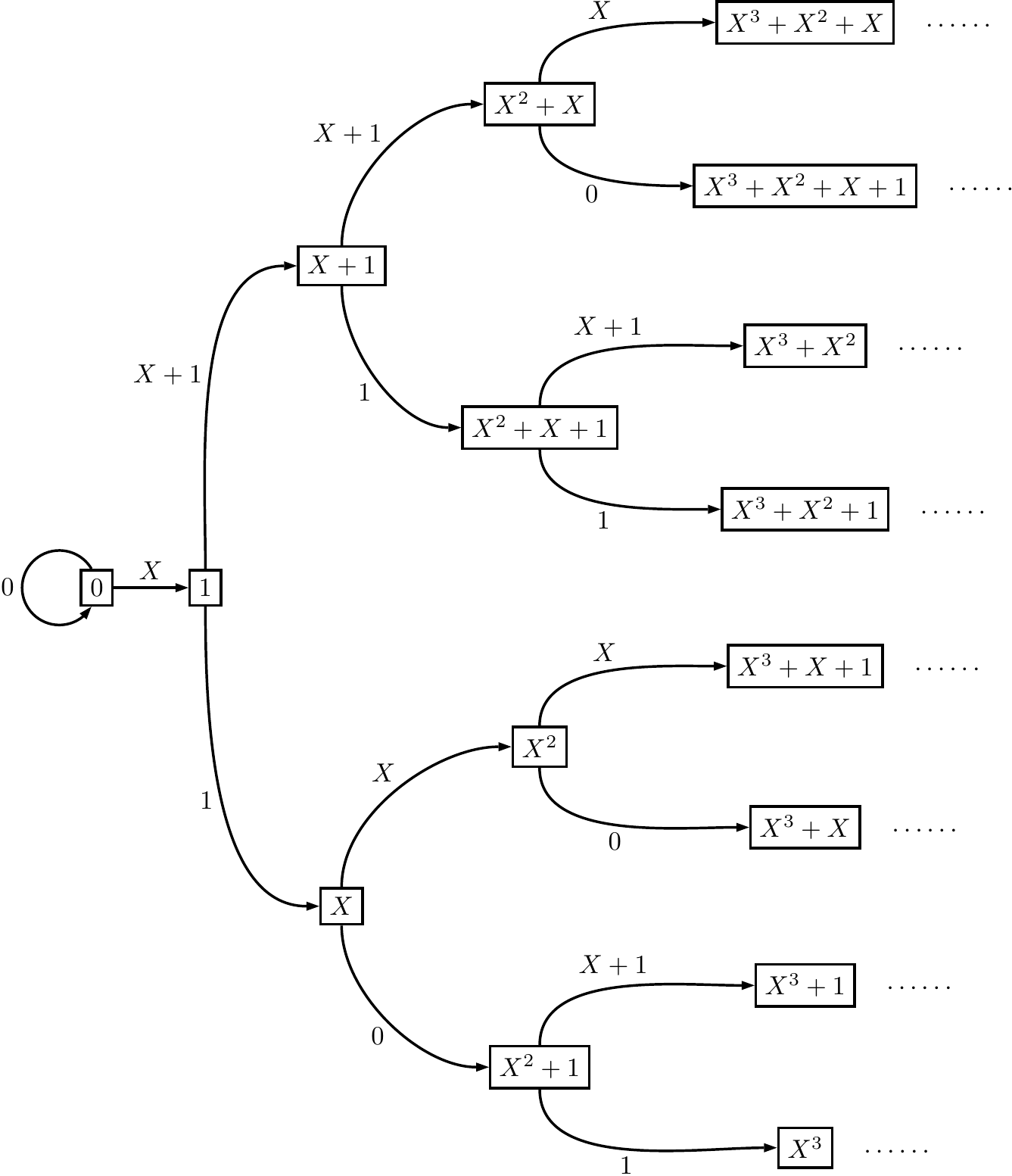}	
			\caption{The graph $T(P/Q)$, where ${P=X^2+1}, {Q=X}\in\F_2[X]$.}\label{tree}
			\end{figure}
		\end{example}		
		
		We define a lexicographic order $\sqsubseteq$ on $\mathcal{L}_{P/Q}$, namely, if $s,t\in\mathcal{L}_{P/Q}$ with $s=s_k\cdots s_0$ and $t=t_{\ell}\cdots t_0$, 
		then $s\sqsubset t$ if $k<\ell$, or if $k=\ell$ and there exists $i$ such that $s_i\sqsubset t_i$ and $s_j=t_j$ for all $i<j\leq k$. 
		The following remark is immediate from~\eqref{e4}.
		
		\begin{remark}\label{treeorder}
			If $v,w\in\F_q[X]$ such that $p(v)\sqsubseteq p(w)$, then $\deg v\le \deg w$.  
		\end{remark}
		
		Remark~\ref{treeorder} implies that the nodes of level $\ell$ in $T(P/Q)$, 
		when arranged with respect to the lexicographic order of their path labels from the root, are ordered according to the partial order on $\F_q[X]$ induced by the degree. 
		However, in general, nodes of the same level are not ordered lexicographically on $\F_q[X]$, see Example~\ref{extree}.

		\begin{definition}[Language of the expansion graph]
			Denote by $\mathcal{D}^{\omega}$ the set of one-sided right-infinite strings on $\mathcal{D}$.
			The subset $W_{P/Q}$ of $\mathcal{D}^{\omega}$ formed by the labels of all infinite paths starting from the root of $T(P/Q)$ 
			is called the \emph{$\omega$-language of the expansion graph $T(P/Q)$}.
		\end{definition}		

		The next result tells us that similar to the language $\mathcal{L}_{P/Q}$, the $\omega$-language $W_{P/Q}$ is not regular 
		(for details and definitions we refer \emph{e.g.} to~\cite{L97,W90}).
		
		\begin{theorem}\label{omegalang}
			Let $(P/Q,\mathcal{D})$ be a rational function based digit system.  
			If $\deg{Q}\geq 1$ then the $\omega$-language $W_{P/Q}$ is not regular.
		\end{theorem}
		\begin{proof}
			Denote by $\operatorname{Pref}(t)$ the set of all finite prefixes of an infinite string $t$.
			By the definition of $W_{P/Q}$, 
			\[\operatorname{Pref}(W_{P/Q})\vcentcolon=\bigcup_{t\in W_{P/Q}}\operatorname{Pref}(t)=0^{\ast}\mathcal{L}_{P/Q}.\]
			Then, it follows from Theorem~\ref{lang} that $\operatorname{Pref}(W_{P/Q})$ is prefix-closed.  
			Moreover, every string $s\in \operatorname{Pref}(W_{P/Q})$ is a proper prefix of another string of $\operatorname{Pref}(W_{P/Q})$ by Theorem~\ref{ktree}.
			Since the language $\mathcal{L}_{P/Q}$ is not regular by Theorem~\ref{lang}, the same holds true for $\operatorname{Pref}(W_{P/Q})$ (\emph{cf.}~\cite[Lemma 5.2.5]{AS}).
			These properties of $\operatorname{Pref}(W_{P/Q})$ imply that the $\omega$-language $W_{P/Q}$ is not regular by~\cite[Proposition 4]{HJ04}.						
		\end{proof}
		
		The following result about the infinite strings of $W_{P/Q}$ is the analogue of~\cite[Proposition 26]{AFS}.
		
		\begin{theorem}\label{eventperiod}
			Let $(P/Q,\mathcal{D})$ be a rational function based digit system and assume that $\deg Q\ge 1$. Then
			the only eventually periodic element of $W_{P/Q}$ is $0^{\omega}\vcentcolon= 00\cdots$.
		\end{theorem}
		\begin{proof}
			Suppose $abb\cdots\in W_{P/Q}$, where $a\in 0^{\ast}\mathcal{L}_{P/Q}$ and $b\in\mathcal{D}^{\ast}$ of length $\ell > 0$.  
			Then for all $n\in\N$, $ab^n\in 0^{\ast}\mathcal{L}_{P/Q}$.  This implies that
			\[\pi(ab^n)-\pi(ab^{n-1})=\left(\frac{P}{Q}\right)^{(n-1)\ell}(\pi(ab)-\pi(a))\in\F_q[X].\]
			Since $P$ and $Q$ are coprime, we have $Q^{(n-1)\ell}$ divides $\pi(ab)-\pi(a)\in\F_q[X]$ for all $n\in\N$. 
			As $\deg Q \ge 1$ it follows from Lemma~\ref{th1} that this is only possible if $abb\cdots=0^{\omega}$.
		\end{proof}

		Recall from Theorem~\ref{firstdigit} that if the $(n+1)$-st element of $\F_q[X]$ 
		according to the lexicographic order is denoted by $a_n$,
		then the sequence $(s^{(0)}(a_n))$ of least significant digits of $P/Q$-polynomial digit expansions of $\F_q[X]$ is the fixed point of a substitution.  
		We now look at the digit sequences of the $P/Q$-polynomial digit expansions of $\F_q[X]$ with respect to a different ordering of elements of $\F_q[X]$.

		Suppose $b_n$ is the $(n+1)$-st element of $\F_q[X]$ arranged according to the lexicographic order on the $(P/Q)$-polynomial digit expansions of elements of $\F_q[X]$.
		It follows from Theorem~\ref{ktree} that the elements of $\F_q[X]$ whose expansions are of length $(\ell +1)$ are 
		$b_{r^{\ell}},\ldots,b_{r^{\ell+1}-1}$, where $r=q^{\deg(P/Q)}$.
		In addition, for all $t\in\N$ and $k\in\set{0,\ldots,r-1}$, we have 
		\begin{equation}\label{bseq}
			b_{rt+k}=\frac{P}{Q}b_t+\frac{s_0}{Q}
		\end{equation}
		for some $s_0\in L(b_t)$.  This implies that 
		$s^{(1)}(b_{rt+k})=s^{(0)}(b_t)$ for all $k\in\set{0,\ldots,r-1}$. 			
		By induction, we obtain that for all $t,m\in\N$ and $k\in\set{0,\ldots,r^m-1}$,
		\[b_{r^mt+k}={\Big(\frac{P}{Q}\Big)}^mb_t+\sum_{i=0}^{m-1}\frac{s_i}{Q}{\Big(\frac{P}{Q}\Big)}^i\]
		for some $s_i\in \mathcal{D}$.
		Thus, we obtain the following result about the digit sequences $(s^{(m)}(b_n))_{n\geq 0}$.
			
		\begin{lemma}\label{otherdigits}
			Let $(P/Q,\mathcal{D})$ be a rational function based digit system and
			$b_n$ be the $(n+1)$-st element of $\F_q[X]$ that are ordered according to the lexicographic order on the $(P/Q)$-polynomial digit expansions of elements of $\F_q[X]$.
			Then for all $t,m\in\N$ and $k\in\set{0,\ldots,r^m-1}$, $s^{(m)}(b_{r^mt+k})=s^{(0)}(b_t)$.
		\end{lemma}
		
		Hence, if the sequence $(s^{(0)}(b_n))_{n\geq 0}$ is known, 
		then each sequence $(s^{(m)}(b_n))_{n\geq 0}$ is generated by repeating each term of $(s^{(0)}(b_n))$ exactly $r^m$ times.		
		Clearly, the first $r$ terms of the sequence $(s^{(0)}(b_n))$ are the elements of 
		\[
		L(0)=\set{dQ\,:\,d\in\F_q[X], \deg{d}<\deg(P/Q)}
		\] 
		ordered by the lexicographic order on $\F_q[X]$.
		From these terms, the first $r$ terms of $(b_n)$ are obtained, namely, they are given by $s^{(0)}(b_0)/Q,\ldots,s^{(0)}(b_{r-1})/Q$.		
		The following proposition tells us how to obtain the rest of the terms of $(s^{(0)}(b_n))$ and of $(b_n)$.
		
		\begin{proposition}\label{bn}
			Let $(P/Q,\mathcal{D})$ be a rational function based digit system.  Then the following assertions hold:
			\begin{enumerate}
				\item If $t_i=\floor{n/r^i}$ for $i\in\set{0,\ldots,m}$, where $m$ is the exponent of the highest power of $r$ that does not exceed $n$, then
				\[
				b_n=\sum_{i=0}^m{\frac{s^{(0)}(b_{t_i})}{Q}{\lt(\frac{P}{Q}\rt)}^i}.\]
				
				\item If $n=rt+k$ where $t\in\N$ and $k\in\set{0,\ldots,r-1}$, then $s^{(0)}(b_n)$ is the $(k+1)$-st element of $L(b_t)$ 
				arranged according to the lexicographic order on $\F_q[X]$.
			\end{enumerate}
		\end{proposition}
		\begin{proof}
			Statement (i) follows from~\eqref{e4} and Lemma~\ref{otherdigits}, while (ii) follows from~\eqref{bseq}.
		\end{proof}
		
		Alternatively, we can obtain the terms of the sequence $(s^{(0)}(b_n))$ by rearranging the terms of the sequence $(s^{(0)}(a_n))$.
		Indeed, by comparing degrees one readily derives that the terms $s^{(0)}(b_{r^{\ell}}),\ldots,s^{(0)}(b_{r^{\ell+1}-1})$ of the sequence $(s^{(0)}(b_n))$ 
		are the terms	$s^{(0)}(a_{r^{\ell}}),\ldots,s^{(0)}(a_{r^{\ell+1}-1})$ of the sequence $(s^{(0)}(a_n))$ in some order.
		
		Recall from Remark~\ref{cobhams} that the sequence $(s^{(0)}(a_n))$ is $q$-automatic.
		The following remark tells us that this property is lost when the terms of $(s^{(0)}(a_n))$ are rearranged to yield the sequence $(s^{(0)}(b_n))$.

		\begin{remark}
			Consider the subsequence $(s^{(0)}(b_{r^n}))_{n\geq 0}$ of the sequence $(s^{(0)}(b_n))$.
			It follows from Proposition~\ref{bn} that 
			\[
			b_{r^{n-1}}=\sum_{i=0}^{n-1}{\frac{s^{(0)}(b_{r^{n-1-i}})}{Q}{\lt(\frac{P}{Q}\rt)}^i}
			\]
			and $s^{(0)}(b_{r^n})\in L(b_{r^{n-1}})$ for all $n\in\N$.
			Hence, the terms of $(s^{(0)}(b_{r^n}))$ are the labels of a nontrivial infinite path starting from the root of the expansion graph $T(P/Q)$.
			It follows then from Theorem~\ref{eventperiod} that $(s^{(0)}(b_{r^n}))$ is not eventually periodic.
			This implies that the sequence $(s^{(0)}(b_n))$ is not $r$-automatic (\emph{cf.}~\cite[Corollary 5.5.3]{AS}).
			Finally, since $r$ is a power of $q$, we conclude that $(s^{(0)}(b_n))$ is not $q$-automatic (\emph{cf.}~\cite[Theorem 6.6.4]{AS}).	
			
			In addition, it is easy to show from above that the sequence $(s^{(m)}(b_n))_{n\geq 0}$ is not $q$-automatic for each $m\ge 1$.				
		\end{remark}
		
	\section{$P/Q$-digit expansions of formal Laurent series}\label{sec:5}
		We have seen in Theorem~\ref{integerRepresentation} that every polynomial over $\F_q$ has a unique finite $P/Q$-polynomial digit expansion.
		In this section, we define $P/Q$-digit expansions for formal Laurent series over $\F_q$. 
		It turns out that these expansions correspond to infinite paths on the expansion graph $T(P/Q)$ defined in Section~\ref{sec:4}. 		
		Let 
		\[\alpha=\sum_{i=-\infty}^{h}{\alpha_iX^i}\;,\; \alpha_h\neq 0,\] 
		be a formal Laurent series over $\F_q$.  
		Then the \emph{degree of $\alpha$} is $\deg{\alpha}=h$, and the (non-Archimedean) \emph{absolute value of $\alpha$} is 
		\begin{equation}\label{absolutevalue}
			|\alpha|=
			\begin{cases}
				q^h, &\text{if }\alpha\neq 0,\\
				0, &\text{if }\alpha=0.
			\end{cases}
		\end{equation}
		The ring $\F_q((X^{-1}))$ of formal Laurent series over $\F_q$ forms a topological ring under the metric induced by the absolute value.		
		Given $r>0$, the disk in $\F_q((X^{-1}))$ centered at $\alpha$ with radius $r$ is     
		\[D(\alpha,r):=\set{\xi\in\F_q((X^{-1})):|\alpha-\xi|<r}.\]
		The \emph{integer part of $\alpha$} and the \emph{fractional part of $\alpha$} are given by 
		\[\floor{\alpha}=\sum_{i=0}^h{\alpha_iX^i}\quad\text{and}\quad\set{\alpha}=\sum_{i=-\infty}^{-1}{\alpha_iX^i},\]
		 respectively. Clearly, $\floor{\alpha}\in\F_q[X]$, $\set{\alpha}\in D(0,1)$, and $\alpha=\floor{\alpha}+\set{\alpha}$.
		
		\begin{remark}
			Given $w\in\F_q[X]$ with $\expan{w}_{P/Q}=(s_k\cdots s_0)_{P/Q}$, 
			we can reformulate~\eqref{e3} as $s_i=P\set{(Q/P)w_i}$ and $w_{i+1}=\floor{(Q/P)w_i}$.
			Hence, if we define the transformation $T$ on $\F_q[X]$ by \[T(w)=\floor{\frac{Q}{P}w},\]
			we have \[s_i=QT^i(w)\imod{P}=P\set{\frac{Q}{P}T^i(w)},\]
			for $0\leq i\leq k$.
		\end{remark}
	
	  We now consider one-sided right-infinite strings over $\mathcal{D}$ that are preceded by a radix point, and we write 
		$\mathcal{D}^{-\omega}=\set{.s_{-1}s_{-2}\cdots \,:\,s_i\in \mathcal{D}}$.
		Using this notation we can extend the evaluation map defined in \eqref{evalmap} as follows:
	  \begin{equation*}
	  \pi:\mathcal{D}^{\ast}\times\mathcal{D}^{-\omega} \rightarrow\F_q((X^{-1}))\;,\;
			s_k\cdots s_0.s_{-1}s_{-2}\cdots\;\mapsto\sum_{i=-\infty}^k\frac{s_i}{Q}{\lt(\frac{P}{Q}\rt)}^i.
		\end{equation*}
		Given $t\in\N$, we have
		${\lt({P}/{Q}\rt)}^t\pi(s_k\cdots s_0.s_{-1}s_{-2}\cdots)=\pi(s_k\cdots s_{-t}.s_{-t-1}\cdots)$, 
		hence,
		$\pi(s_k\cdots s_0.s_{-1}s_{-2}\cdots)=\lim_{t\rightarrow\infty}{\lt({P}/{Q}\rt)}^{-t}\pi(s_{k}\cdots s_{-t})$.
		In addition, $\pi$ is continuous if $\mathcal{D}^{\ast}\times\mathcal{D}^{-\omega}$ carries the usual metric, 
		{\it i.e.}, two digit strings are considered to be close when they disagree only on digits with large negative indices.
		
		It turns out that infinite strings in $W_{P/Q}$ are of particular importance.
		
		\begin{definition}[{$P/Q$-digit expansion in $\F_q{((X^{-1}))}$}]\label{def_LS}
			Let $(P/Q,\mathcal{D})$ be a rational function based digit system and $\alpha\in \F_q{((X^{-1}))}$. 
			If $\alpha=\pi(s_k\cdots s_0.s_{-1}s_{-2}\cdots)$ for $s=s_k\cdots s_0s_{-1}s_{-2}\cdots\in W_{P/Q}$, 
			then we say that $\pi(s_k\cdots s_0.s_{-1}s_{-2}\cdots)$
			is a \emph{$P/Q$-(series) digit expansion} of the formal Laurent series $\alpha$  with \emph{digit string} $s$.
		\end{definition}
		
		\begin{remark}
			The $P/Q$-digit expansions we have defined here are not the same as
			the \emph{$\beta$-expansions} of formal Laurent series over $\F_q$ that were 
			introduced independently in~\cite{HM06,S07}.  
			To be more precise, if $\beta=P/Q$, then the $\beta$-expansion of $\alpha\in D(0,1)$ is given by 
			$d_{\beta}(\alpha)=.d_{-1}d_{-2}\cdots$, where 
			\[\alpha=\sum_{i=-\infty}^{-1}{d_i{\lt(\frac{P}{Q}\rt)}^i}\]
			and the sequence $(d_i)_{-\infty<i\leq -1}$ is obtained by the greedy algorithm.						
			Indeed, the factor $1/Q$ is missing. 
			Furthermore, the digits $d_i$ are polynomials of degree less than $\deg(P/Q)=\deg{P}-\deg{Q}$, 
			whereas the digits in $P/Q$-digit expansions are allowed to have degree up to $\deg{P}-1$.
		\end{remark}		
		
		It turns out that restricting the digit strings to $W_{P/Q}$ in Definition~\ref{def_LS} leads to a unique $P/Q$-digit expansion for each formal Laurent series over $\F_q$.
		
		\begin{theorem}\label{wpq}
			Let $(P/Q,\mathcal{D})$ be a rational function based digit system.
			Then every $\alpha \in \F_q((X^{-1}))$ has a unique $P/Q$-digit expansion
			\begin{equation}\label{laurserexp}
				\alpha =\sum_{i=-\infty}^{k}\frac{s_i}{Q}{\lt(\frac{P}{Q}\rt)}^{i}.
			\end{equation}
		\end{theorem}
		
		Let $(P/Q,\mathcal{D})$ be a rational function based digit system, and set $m\vcentcolon=\deg{P}$, $n\vcentcolon=\deg{Q}$, and $r\vcentcolon=\deg(P/Q)=m-n$. 
		To establish Theorem~\ref{wpq}, we need the following lemmas.  
					
		\begin{lemma}\label{degreeint}
			Let $s\in\mathcal{D}$ and $i\in\Z$.
			\begin{enumerate}
				\item If $n\leq \deg{s}<m$, then $-ri\leq\deg{(s/Q){(P/Q)}^{-i}}<-r(i-1)$.  
				
				\item If $0\leq \deg{s}<n$, then $-ri-n\leq\deg{(s/Q){(P/Q)}^{-i}}<-ri$.
				
				\item If $s=0$, then $\deg{(s/Q){(P/Q)}^{-i}}=-\infty$.
			\end{enumerate}
		\end{lemma}
		\begin{proof}
				Trivial.
		\end{proof}
				
		In the following lemma we use the set $L(v)$ of labels of the outgoing edges from $v\in\F_q[X]$ in $T(P/Q)$.
				  
		\begin{lemma}\label{edge}
			Choose $v\in\F_q[X]$ and $\beta\in\F_q((X^{-1}))$ with $\deg{\beta}<m$.  Then there exists a unique $s\in L(v)$ such that $\deg(\beta-s)<n$.
		\end{lemma}
		\begin{proof}
			Let $t=t_0+t_1X+\cdots+t_{m-1}X^{m-1}$ and $u=u_0+u_1X+\cdots+u_{m-1}X^{m-1}$ 
			be distinct digits in $L(v)$.
			It follows from ~\eqref{edgelabels} that $t-u=dQ$ for some nonzero $d\in\F_q[X]$.
			Thus, $\deg(t-u)\geq n$ which means that
			$t_nX^n+\cdots+t_{m-1}X^{m-1}$ is different from $u_nX^n+\cdots+u_{m-1}X^{m-1}$.
			Hence, each digit in $L(v)$ corresponds to a unique polynomial in $\F_q[X]$ of the form $a_{n}X^n+\cdots+a_{m-1}X^{m-1}$.
			However, the number of polynomials of this form is exactly $|L(v)|$ by Theorem~\ref{ktree}.
			Therefore, if $\beta=\sum_{i=-\infty}^{m-1}{\beta_iX^i}$, 
			then there is a unique $s=s_0+s_1X+\cdots+s_{m-1}X^{m-1}\in L(v)$ such that $s_i=\beta_i$ for $n\leq i<m$.
			This proves the claim.				
		\end{proof}
		
		We are now ready to prove the theorem.
		
		\begin{proof}[Proof of Theorem~\ref{wpq}]
			Let $\alpha$ be a formal Laurent series over $\F_q$. 
			For the moment we assume that $\lfloor \alpha \rfloor=0$ and write $\alpha=\sum_{i=-\infty}^{-1}{\alpha_iX^i}$.
			We will now determine the unique $P/Q$-digit expansion $s_k\cdots s_0.s_{-1}s_{-2}\cdots$ for $\alpha$. 
			
			Assume first that there exists a $P/Q$-digit representation $s_k\cdots s_0.s_{-1}s_{-2}\cdots$ for $\alpha$ with $s_k\neq 0$ 
			for some $k\ge 0$. Then $s_k$ has to be a nonzero digit in $L(0)$ and, hence, $\deg s_k \ge n$ by~\eqref{edgelabels}. 
			This implies that $\deg{(s_k/Q){(P/Q)}^k}\ge rk$ by Lemma~\ref{degreeint}~(i). 
			Since $\deg{(s_i/Q){(P/Q)}^i}<rk$ holds for $i<k$ by the same lemma, the leading term of the sum
			\begin{equation}\label{eq:toobig}
				\sum_{i=-\infty}^{k}\frac{s_i}{Q}{\lt(\frac{P}{Q}\rt)}^{i}
			\end{equation}
			cannot cancel with any of the other terms and, thus, the sum has degree at least $rk \ge 0$. 
			However, $\deg \alpha \le -1$ and so the sum in \eqref{eq:toobig} cannot be a $P/Q$-digit expansion of $\alpha$. 
			Thus, if there exists a $P/Q$-digit expansion for $\alpha$ it has to be of the form $0.s_{-1}s_{-2}\cdots$. 
			The remaining part of the proof is devoted to the description of an algorithm that determines the unique $P/Q$-digit expansion of $\alpha$ of this form.	 
			
			Set $\alpha^{(1)}=\alpha$ and $v_1=0$.
			We see from Lemma~\ref{degreeint} that the prefix $\alpha_{-1}X^{-1}+\cdots+\alpha_{-r}X^{-r}$ of $\alpha^{(1)}$ 
			is determined only by the term $(s_{-1}/Q){(P/Q)}^{-1}=s_1/P$ of~\eqref{laurserexp} since all the summands of~\eqref{laurserexp} with smaller index have degree less than $-r$. 
			We set $\beta_1=P(\alpha_{-1}X^{-1}+\cdots+\alpha_{-r}X^{-r})$.
			Since $\deg{\beta_1}<m$, we have to choose $s_{-1}$ as the unique element of $L(v_1)$ such that $\deg(\beta_1-s_{-1})<n$ by Lemma~\ref{edge} 
			as this is the unique digit in $T(P/Q)$ on an outgoing edge from $v_1=0$ that guarantees that $\alpha^{(1)}-(s_{-1}/Q){(P/Q)}^{-1}$ has degree less than $-r$. 
			In view of Lemma~\ref{degreeint} this is necessary to represent $\alpha^{(1)}-(s_{-1}/Q){(P/Q)}^{-1}$ by a sum of the form \eqref{eq:toobig} with $k=-2$ and we can go on.
			
			Thus, we have uniquely determined a first digit $s_{-1}\in L(v_1)$ such that 
			the first $r$ terms of $(s_{-1}/Q){(P/Q)}^{-1}$ are precisely $\alpha_{-1}X^{-1}+\cdots+\alpha_{-r}X^{-r}$. 
			However, the remaining terms of $(s_{-1}/Q){(P/Q)}^{-1}$ may be different from those of $\alpha^{(1)}$. 
			This means that we now have to consider
			\[\alpha^{(2)}=\alpha^{(1)}-\frac{s_{-1}}{Q}{\lt(\frac{P}{Q}\rt)}^{-1}=\sum_{i=-\infty}^{-r-1}{\alpha_i^{(2)}X^i}.\] 
			Note that the node adjacent to $v_1=0$ by the edge with label $s_{-1}$ in $T(P/Q)$ is $v_2=(P/Q)v_1+s_{-1}/Q=s_{-1}/Q$.  
			Similarly, the prefix $\alpha_{-r-1}^{(2)}X^{-r-1}+\cdots+\alpha_{-2r}^{(2)}X^{-2r}$ of $\alpha^{(2)}$ 
			is determined only by the term $(s_{-2}/Q){(P/Q)}^{-2}$ of~\eqref{laurserexp} since all the summands of~\eqref{laurserexp} with smaller index have degree less than $-2r$.
			So we look at 
			\[\beta_2=Q{\lt(\frac{P}{Q}\rt)}^2(\alpha_{-r-1}^{(2)}X^{-r-1}+\cdots+\alpha_{-2r}^{(2)}X^{-2r}).\]
			We see that $\deg{\beta_2}<m$ and, thus, Lemma~\ref{edge} guarantees that there is a unique element $s_{-2}\in L(v_2)$ such that $\deg(\beta_2-s_{-2})<n$. 
			This is the unique digit in $T(P/Q)$ on an outgoing edge from $v_2$ that ensures that
			\begin{equation}\label{a2}
				\alpha^{(2)}-\frac{s_{-2}}{Q}{\lt(\frac{P}{Q}\rt)}^{-2}=\alpha^{(1)}-\frac{s_{-1}}{Q}{\lt(\frac{P}{Q}\rt)}^{-1}-\frac{s_{-2}}{Q}{\lt(\frac{P}{Q}\rt)}^{-2}
			\end{equation}
			has degree less than $-2r$.  
			In view of Lemma~\ref{degreeint} this is necessary to represent~\eqref{a2} by a sum of the form~\eqref{eq:toobig} with $k=-3$ and we can continue.
			
			We continue recursively.  In particular, for $j\in\N$, we compute
			\begin{equation}
				\begin{aligned}
					\alpha^{(j)}&=\alpha^{(j-1)}-\frac{s_{-(j-1)}}{Q}{\lt(\frac{P}{Q}\right)}^{-(j-1)}=\sum_{i=-\infty}^{-(j-1)r-1}{\alpha_i^{(j)}X^{i}},\\
					v_j&=\frac{P}{Q}v_{j-1}+\frac{s_{-(j-1)}}{Q},\text{ and}\\
					\beta_j&=Q{\lt(\frac{P}{Q}\rt)}^j(\alpha_{-(j-1)r-1}^{(j)}X^{-(j-1)r-1}+\cdots+\alpha_{-jr}^{(j)}X^{-jr}).		
				\end{aligned}\label{rec2}
			\end{equation}
			Since $\deg{\beta_j}<m$, there is a unique $s_{-j}\in L(v_j)$ for which $\deg(\beta_j-s_{-j})<n$ by Lemma~\ref{edge}. 
			This is the unique digit in $T(P/Q)$ occurring on an edge leading out of $v_j$ that ensures that $\alpha^{(j)} - (s_{-j}/Q){(P/Q)}^{-j}$ has degree less than $-jr$. 
			In view of Lemma~\ref{degreeint} this is necessary to represent $\alpha^{(j)} - (s_{-j}/Q){(P/Q)}^{-j}$ by a sum of the form \eqref{eq:toobig} with $k=-(j+1)$ and we can go on.
						
			From this recursion, we generate an infinite string $s_{-1}s_{-2}\cdots$ that satisfies $s_{-1}s_{-2}\cdots\in W_{P/Q}$ and $\pi(0.s_{-1}s_{-2}\cdots)=\alpha$. 
			Moreover, during the generation of this string it was shown that we are forced to choose each digit in this representation in a unique way 
			that is imposed by the constraints of the language $W_{P/Q}$. 
			This yields the uniqueness of this representation.

			For general $\alpha\in\F_q((X^{-1}))$ it suffices to multiply $\alpha$ by a suitable negative power $(P/Q)^{-k}$ such that $\floor{(P/Q)^{-k}\alpha}$ has zero integer part.  
			Since $T(P/Q)$ has a loop at $0$, this multiplication just shifts the radix point of the $P/Q$-digit expansion of $\alpha$ and, hence, 
			the theorem is also proved for the general case.
		\end{proof}	
		
		\begin{remark}
			The $\omega$-language formed by the digit strings associated with $P/Q$-digit expansions of formal Laurent series over $\F_q$ is precisely $W_{P/Q}$.
			This $\omega$-language is not regular by Theorem~\ref{omegalang}.
		\end{remark}
		
		Let $(P/Q,\mathcal{D})$ be a rational function based digit system. 
		The unique (in view of Theorem~\ref{wpq}) $P/Q$-digit expansion of  $\alpha\in \F_q{((X^{-1}))}$ will be denoted by 
		$\sexpan{\alpha}_{P/Q}=(s_k\cdots s_0.s_{-1}s_{-2}\cdots)_{P/Q}$. 
		For the sake of simplicity of notation, we will sometimes write $s_k\cdots s_0.s_{-1}s_{-2}\cdots$.  

		\pagebreak\begin{remark}\mbox{}
		\begin{enumerate}
				\item Recall that every polynomial $w$ in $\F_q[X]$ has a unique $P/Q$-polynomial digit expansion by Theorem~\ref{integerRepresentation}. 
				On the other hand, $w$ viewed as a formal Laurent series also has a unique $P/Q$-series digit expansion by Theorem~\ref{wpq}.  
				Clearly, these two digit expansions are different because the former is finite while the latter is infinite (see Example~\ref{fls}).
				
				\item As illustrated in Example~\ref{fls}, $\sexpan{\alpha}_{P/Q}= s_k\cdots s_0.s_{-1}s_{-2}\cdots$ implies that $\expan{\floor{\alpha}}_{P/Q}=s_k\cdots s_0$. 
				However, in general $\sexpan{\set{\alpha}}_{P/Q}=.s_{-1}s_{-2}\cdots$ does not hold even if $\pi(.s_{-1}s_{-2}\cdots)=\set{\alpha}$ remains true. 
			\end{enumerate}
		\end{remark}	
		
		\begin{example}\label{fls}
			Choose $P=X^2+1$ and $Q=X$ and consider $\alpha\in\F_2((X^{-1}))$ with
			\[\alpha= X + 1 + \sum_{\underset{\sst i\not\equiv 0\imod{3}}{i<0}}{X^i}.\]
			Using the algorithm established in the proof of Theorem~\ref{wpq} we gain 
			\[\sexpan{\alpha}_{P/Q}=X,X+1.X+1,0,X+1,X,X+1,X+1,1,0,1,1,\cdots,\]
			where digits are separated by commas for the sake of readability.
			Let us now look at the $P/Q$-polynomial digit expansion and the $P/Q$-series digit expansion of $\floor{\alpha}$.
			Using~\eqref{e3} for the former and the above algorithm for the latter yield
			\begin{align*}
				\expan{\floor{\alpha}}_{P/Q}&=X,X+1,\quad\text{ and}\\
				\sexpan{\floor{\alpha}}_{P/Q}&=X,X+1.1,X+1,X,X+1,0,0,X,X+1,0,0,\cdots.
			\end{align*}	
			 In addition, the $P/Q$-representation of $\set{\alpha}$ is given by
			\[\sexpan{\set{\alpha}}_{P/Q}=.X,X+1,1,1,X+1,X+1,X+1,X+1,1,1,\cdots.\]
		\end{example}				
		
		\begin{remark}\label{rem:powerseries}
			At a first glance it might seem to be more natural to extend $P/Q$-expansions to the ring of formal power series $\F_q[[X]]$ rather than to the field $\F_q((X^{-1}))$. 
			However, due to the fact that the language $\mathcal{L}_{P/Q}$ is not suffix closed, 
			the analogue of the graph $T(P/Q)$ cannot be defined in a natural way that leads to unique expansions. 
			The only reasonable definition would be to go on with the full shift which allows infinitely many different representations for each element.
		\end{remark}
		
	\section{Explicit formulas for $P/Q$-digit expansions}\label{sec:6a}
		In this section we will provide formulas for the $P/Q$-polynomial digit expansions of polynomials and the $P/Q$-digit expansions of formal Laurent series. 
		Formulas of a similar kind have been worked out for $\beta$-expansions of formal Laurent series over finite fields in~\cite{SS}. 
		Even if we use some of the methods developed there, the situation here is more intricate as the expansion graph $T(P/Q)$ has a very difficult structure. 
		Indeed, new ideas are required in the proofs of Theorems~\ref{digitpoly} and~\ref{digitLaurent} below.
			
		We will now represent digit expansions in some larger domain.  
		Indeed, for the indeterminate $b$ ($b$ stands for basis) we consider the rings $\F_q(X)[b]$ and $\F_q((X^{-1},b^{-1}))$ in two variables. 
		First, let $w\in\F_q[X]$ with $\expan{w}_{P/Q}=s_k\cdots s_0$ be given. 
		In view of Theorem~\ref{integerRepresentation}, $w$ has a unique $P/Q$-polynomial digit expansion given by~\eqref{e4}.
		This allows us to define a function $\pdb:\F_q[X]\rightarrow\F_q(X)[b]$ by
		\begin{equation}\label{eq:elementpoly}
			\pdb(w)=\sum_{i=0}^k\frac{s_i}{Q}\,b^i.
		\end{equation}	
		Similarly, if $\alpha\in\F_q((X^{-1}))$ with $\sexpan{\alpha}_{P/Q}=s_k\cdots s_0.s_{-1}s_{-2}\cdots$ 
		then it follows from Theorem~\ref{wpq} that $\alpha$ can be represented uniquely by its $P/Q$-digit expansion
		\begin{equation}\label{eq:exp}
			\alpha=\sum_{i=-\infty}^k\frac{s_i}{Q}{\lt(\frac{P}{Q}\rt)}^i.
		\end{equation}
		This again allows us to define a function $\db:\F_q((X^{-1}))\rightarrow\F_q((X^{-1},b^{-1}))$ by
		\begin{equation}\label{eq:element}
			\db(\alpha)=\sum_{i=-\infty}^k\frac{s_i}{Q}\,b^i.
		\end{equation}
		Note that~\eqref{e4} and ~\eqref{eq:exp} are just expansions of $w$ and $\alpha$, respectively. 
		On the other hand, \eqref{eq:elementpoly} and~\eqref{eq:element} are elements of the rings $\F_q(X)[b]$ and $\F_q((X^{-1},b^{-1}))$, respectively. 
		Nonetheless, knowing the element $\pdb(w)$ is the same as knowing the $P/Q$-polynomial digit expansion of $w$. 
		At the same time, knowing the element $\db(\alpha)$ is the same as knowing the $P/Q$-digit expansion of $\alpha$. 
		Thus finding the $P/Q$-polynomial digit expansion of $w$ and the $P/Q$-digit expansion of $\alpha$ is equivalent 
		to finding formulas for $\pdb(w)$ and $\db(\alpha)$, respectively. 
		Without risk of confusion we also call $\pdb(w)$ the $P/Q$-polynomial digit expansion of $w$, and $\db(\alpha)$ the $P/Q$-digit expansion of $\alpha$.
			
		We will also need a kind of inverse for $\pdb$ and for $\db$, namely, 
		we would like to ``evaluate'' elements of $\F_q(X)[b]$ and $\F_q((X^{-1},b^{-1}))$ by mapping $b$ to $P/Q$. 
		More precisely, we define the map
		\[\operatorname{eval_{pd}}:\F_q(X)[b]\rightarrow\F_q(X)\] 
		where $b$ is mapped to $P/Q$ for a given element of $\F_q(X)[b]$.
		In the same manner, the map
		\[\operatorname{eval_d}:\F_q((X^{-1},b^{-1}))\rightarrow\F_q((X^{-1}))\] 
		is defined by mapping $b$ to $P/Q$ in a formal Laurent series of $\F_q((X^{-1},b^{-1}))$. 
				
		The following lemma is straightforward, and the succeeding lemma follows from Theorem~\ref{integerRepresentation} and Theorem~\ref{wpq}.
			
		\begin{lemma}\label{lem:equiv}	
			Let $f=Qb-P$.
			\begin{enumerate}
				\item If $y,y'\in\F_q(X)[b]$ with $y-y'\in f\,\F_q(X)[b]$, then 
				\[\operatorname{eval_{pd}}(y)=\operatorname{eval_{pd}}(y').\]
				
				\item If $\psi,\psi'\in\F_q((X^{-1},b^{-1}))$ with $\psi-\psi'\in f\,\F_q((X^{-1},b^{-1}))$, then 
				\[\operatorname{eval_d}(\psi)=\operatorname{eval_d}(\psi').\]
			\end{enumerate}
		\end{lemma}
			
		\begin{lemma}\label{lem:72}
			Let $(P/Q,\mathcal{D})$ be a rational function based digit system. 
			\begin{enumerate}
				\item If $w\in\F_q[X]$, then 
				\[y=\sum_{i=0}^k{\frac{s_i}{Q}\,b^i}\in\F_q(X)[b]\] 
				satisfies $y=\pdb(w)$ if and only if $\operatorname{eval_{pd}}(y)=w$ and $s_i\in\mathcal{D}$ for $0\leq i\leq k$.
				
				\item If $\alpha\in\F_q((X^{-1}))$, then 
				\[\psi=\sum_{i=-\infty}^k\frac{s_i}{Q}\,b^i\in\F_q((X^{-1},b^{-1}))\] 
				satisfies $\psi=\db(\alpha)$ if and only if $\operatorname{eval_d}(\psi)=\alpha$ and the infinite string $s_k\cdots s_0s_{-1}s_{-2}\cdots\in W_{P/Q}$.
			\end{enumerate}
		\end{lemma}		
			
		Thus, to obtain $\pdb(w)$ for some $w\in\F_q[X]$ we could start with an arbitrary representative $w' \in \F_q(X)[b]$ 
		satisfying $\operatorname{eval_{pd}}(w)=\operatorname{eval_{pd}}(w')$ (which can be just the natural embedding of $w$ in $\F_q(X)[b]$). 
		Then $\pdb(w)$ can be computed by adding a suitable $\F_q(X)[b]$-multiple of $f$ to $w'$. 
		The same procedure applies in computing $\db(\alpha)$ for a given $\alpha\in\F_q((X^{-1}))$.
		That is, start with $\alpha'\in\F_q((X^{-1},b^{-1}))$ satisfying $\operatorname{eval_d}(\alpha)=\operatorname{eval_d}(\alpha')$ 
		and add a suitable $\F_q((X^{-1},b^{-1}))$-multiple of $f$ to $\alpha'$. 
		
		To establish the desired formulas for $\pdb(w)$ and $\db(\alpha)$ we have to find the suitable multiple of $f$ in both cases. 
		To this end we need to compute multiplicative inverses of some elements in $\F_q((X^{-1},b^{-1}))$. 
		For this we define the notion of an $\A$-Laurent series (see also~\cite{SS}). 
		Let $\mathbf{x},\mathbf{y}\in\Z^2$ be linearly independent.
		Set
		\begin{equation}\label{def:A}
			\A\vcentcolon=\set{\mathbf{a}\,:\,\mathbf{a}=s\mathbf{x}+t\mathbf{y},\;s,t\geq 0}\cap\Z^2.
		\end{equation}
		If $\m\in\Z^2$, then we define $\A+\m$ as the Minkowski sum 
		\[
		\A+\m\vcentcolon=\set{\mathbf{a}+\m\,:\,\mathbf{a}\in\A}.
		\]
		The commutative ring $\F_q[[\A]]$ of \emph{$\A$-power series} in the indeterminates $b$ and $X$ is defined by 
		\[
		\F_q[[\A]]\vcentcolon=\Bigg\{\sum_{(i,j)\in\A}{a_{ij}b^iX^j}\,:\,a_{ij}\in\F_q\Bigg\}.
		\]
		Moreover, the commutative ring of \emph{$\A$-Laurent series} in the indeterminates $b$ and $X$ is given by
		\[
		\F_q((\A))\vcentcolon=\bigcup_{\mathbf{m}\in \mathbb{Z}^2} \F_q[[\A+m]].
		\]
		Suppose $\alpha\in\F_q((\A))$ is written in the form 
		\begin{equation}\label{ALaurent}
			\alpha=\sum_{(i,j)\in \A+\m}{a_{ij}b^iX^j}.
		\end{equation}
		Then the \emph{support} of $\alpha$ is defined as
		\[\supp(\alpha)\vcentcolon=\set{(i,j)\in\Z^2\,:\,a_{ij}\neq 0}.\]
		If $\alpha\in\F_q((\A))$ and there exists $\mathbf{m}\in\supp(\alpha)$ such that $\supp(\alpha)\subset\mathcal{A}+\mathbf{m}$, 
		then it is immediate that $\mathbf{m}$ is unique with this property. 
		In such a case we say that $\alpha$ has an \emph{anchor} and write $\mathbf{m}=\anc_{\A}(\alpha)$.
	
		Let $\mathcal{A}$ be defined as in \eqref{def:A}. Following~\cite[Definition~4.4]{SS}, let $\langle\cdot,\cdot\rangle$ be the Euclidean inner product of $\R^2$
		and choose $\mathbf{v}\in\Z^2$ with $\langle\mathbf{x},\mathbf{v}\rangle>0$ and $\langle\mathbf{y},\mathbf{v}\rangle>0$.
		Then $\langle\mathbf{v},\mathbf{z}\rangle\geq 0$ for all $\mathbf{z}\in\A$.
		For $\alpha\in\F_q((\A))$, let
		\[\operatorname{ord}_\A(\alpha)\vcentcolon=\min_{\mathbf{z}\in\supp(\alpha)}\langle\mathbf{v},\mathbf{z}\rangle
		\quad\hbox{and}\quad |\alpha|\vcentcolon=q^{-\operatorname{ord}_{\A}(\alpha)}.\]
		As $\A$ is discrete, this is well defined and makes $\F_q((\A))$ a topological ring.  
		As mentioned in \cite{SS}, $\F_q((\A))$ is not complete as a topological space. 
		However, the following weaker result holds.

		\begin{lemma}[{{\it cf.}~\cite[Lemma~4.5]{SS}}]\label{limitexists}
			If $(\alpha_k)_{k\geq 0}$ is a sequence in $\F_q((\A))$ which is Cauchy and $\supp(\alpha_k)\subseteq \A+\mathbf{m}$ for a fixed $\mathbf{m}\in\Z^2$ and all $k\geq 1$,
			then $\lim_{k\rightarrow\infty}\alpha_k$ exists in $\F_q((\A))$.
		\end{lemma}

		Although the anchor was defined in a slightly different way in \cite{SS}, the following lemma can be proved in the same manner as \cite[Lemma 4.10]{SS}. 	
					
		\begin{lemma}\label{units}
			If $\alpha\in\F_q((\A))$, then $\alpha$ is a unit of $\F_q((\A))$ if and only if $\anc_\A(\alpha)$ exists.
		\end{lemma}
			
		Suppose $P,Q\in\F_q[X]$, with $m\vcentcolon=\deg{P}$, $n\vcentcolon=\deg{Q}$, and $0\leq n<m$.
		We shall now consider the polynomial $f=Qb-P$.
		Denote by $\mathbf{P}$ and $\mathbf{Q}$ the points $(0,m)$ and $(1,n)$, respectively, on the $(b,X)$-plane (see Figure~\ref{supp}).
		
		\begin{figure}[ht]
			\includegraphics{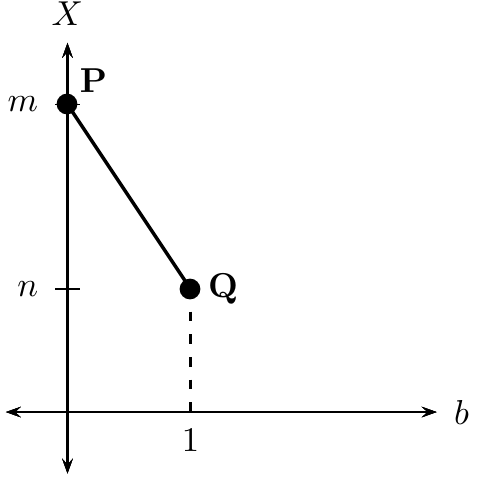}	
			\caption{Anchors of $f=Qb-P$}\label{supp}
		\end{figure}		
	
		If $\A_1$ is the cone in $\Z^2$ defined by
		\[\A_1=\set{s(0,-1)+t(1,n-m)\,:\,s,t\geq 0}\cap\Z^2,\]
		then $f\in\F_q((\A_1))$.
		Since the anchor of $f$ exists with $\anc_{\A_1}(f)=\mathbf{P}$, it follows from Lemma~\ref{units} that $f$ is a unit of $\F_q((\A_1))$.
		Let $p_m$ be the leading coefficient of $P$ and 
		\[f=Qb-P=-p_m X^m(1-R)\;\;\Longleftrightarrow\;\;R=1+\frac{f X^{-m}}{p_m}.\]
		Then the multiplicative inverse of $f$ in $\F_q((\A_1))$ is given by
		\begin{equation}\label{invP}
			f_{\mathbf{P}}^{-1}=\frac{-X^{-m}}{p_m}\sum_{i=0}^{\infty}R^i=\frac{-X^{-m}}{p_m}\sum_{i=0}^{\infty}{\lt(1+\frac{fX^{-m}}{p_m}\rt)}^i.
		\end{equation}
		By Lemma~\ref{limitexists} this sum converges since the partial sums form a Cauchy sequence with support in $\A_1$.
		Analogously, if $\A_2$ is defined as 
		\[\A_2=\set{s(0,-1)+t(-1,m-n)\,:\,s,t\geq 0}\cap\Z^2,\] 
		it follows that $f\in\F_q((\A_2))$ with $\anc_{\A_2}(f)=\mathbf{Q}$. 
		Applying the same argument as before, it follows that $f$ is a unit of $\F_q((\A_2))$ with 
		\begin{equation}\label{invQ}
			f_{\mathbf{Q}}^{-1}=\frac{X^{-n}b^{-1}}{q_n}\sum_{i=0}^{\infty}{\lt(1-\frac{fX^{-n}b^{-1}}{q_n}\rt)}^i,
		\end{equation}
		where $q_n$ is the leading coefficient of $Q$. 			
			
		If $\alpha$ is defined as in~\eqref{ALaurent}, we denote by
		\[\floor{\alpha}_b=\sum_{\begin{subarray}{c}(i,j)\in\supp(\alpha)\\i\geq 0\end{subarray}}{a_{ij}b^iX^j}\quad\hbox{and}\quad\set{\alpha}_b=\alpha-\floor{\alpha}_b.\]
		The notations $\floor{\alpha}_X$ and $\set{\alpha}_X$ are defined analogously.
		
		We need the following auxiliary lemma.
		
		\begin{lemma}\label{lem:cut}\mbox{}
			\begin{enumerate}
				\item Let $w\in\F_q[X]$ be given and $f_{\mathbf{P}}^{-1}$ as in \eqref{invP}. Then $\floor{w f_{\mathbf{P}}^{-1}}_X\in\F_q[b,X]$.
			
				\item Let $v\in\F_q[b](X)$ be given and $f_{\mathbf{Q}}^{-1}$ as in \eqref{invQ}. Then $\floor{\smash{v f_{\mathbf{Q}}^{-1}}}_b\in\F_q[b](X)$.
			\end{enumerate}
		\end{lemma}			
		\begin{proof}
			It follows from~\eqref{invP} that $\supp(f_{\mathbf{P}}^{-1})$ is contained in the cone $C$ with vertex $(0,-m)$ bounded by the rays $(0,-1)$ and $(1, n-m)$. 
			Thus $\supp(w f_{\mathbf{P}}^{-1})$ is contained in $C+(\deg(w),0)$. 
			Thus, as $\floor{\cdot}_X$ cuts off the part of the support that is contained in the lower half plane, 
			$\supp(\floor{w f_{\mathbf{P}}^{-1}}_X)$ is either empty or contained in a finite triangle contained in the first quadrant. 
			This proves (i). Assertion (ii) is proved in a similar way.
 		\end{proof}
		
		The next theorem gives an explicit formula to compute $\pdb(w)$ for a given polynomial $w$ over $\F_q$, and vice versa.	
				
		\begin{theorem}\label{digitpoly}
			Let $(P/Q,\mathcal{D})$ be a rational function based digit system and set $f=Qb-P$.
			If $w\in\F_q[X]$, then
			\begin{align}
				\label{f1}
				\pdb(w)&=\tfrac{f}{Q}\set{f_{\mathbf{P}}^{-1}Qw}_X=w-\tfrac{f}{Q}\floor{f_{\mathbf{P}}^{-1}Qw}_X,\text{ and}\\
				\label{f2}
				w&=f\set{\smash{f_{\mathbf{Q}}^{-1}}\pdb(w)}_b=\pdb(w)-f\floor{\smash{f_{\mathbf{Q}}^{-1}}\pdb(w)}_b.
			\end{align}
		\end{theorem}
		\begin{proof}
			First, note that if we define 
			\[
			y\vcentcolon=\tfrac{f}{Q}\set{f_{\mathbf{P}}^{-1}Qw}_X=w-\tfrac{f}{Q}\floor{f_{\mathbf{P}}^{-1}Qw}_X,
			\]
			then we obtain 
			\[w-y=\tfrac{f}{Q}\floor{f_{\mathbf{P}}^{-1}Qw}_X.\]
			Now, Lemma~\ref{lem:cut} yields $(1/Q)\floor{f_{\mathbf{P}}^{-1}Qw}_X\in\F_q[b](X)$ and, hence,
			we may apply Lemma~\ref{lem:equiv}~(i) to obtain $\operatorname{eval_{pd}}(y)=\operatorname{eval_{pd}}(w)=w$.
			
			All that remains is to verify that $y$ is a $P/Q$-polynomial digit expansion. 
			In view of \eqref{eq:elementpoly}, we need to check that $Qy$ contains no negative power of $b$ and contains only powers $X^i$ of $X$ satisfying $0\leq i<\deg{P}$. 
			By Lemma~\ref{lem:cut}~(i)
			\begin{equation}\label{eq:pbx}
			Qy=f\set{f_{\mathbf{P}}^{-1}Qw}_X=Qw-f\floor{f_{\mathbf{P}}^{-1}Qw}_X \in \F_q[b,X]
			\end{equation}
			which proves that the conditions on the positivity of the powers of $b$ and $X$ in $Qy$ hold.
			In addition, we have 
			\[
			\deg_X(Qy)=\deg_X\lt(f\set{f_{\mathbf{P}}^{-1}Qw}_X\rt)<\deg{P}
			\]
			showing that the conditions on the powers of $X$ in $Qy$ are also satisfied. 
			It now follows from Lemma~\ref{lem:72} that~\eqref{f1} holds.
			
			Suppose that $\pdb(w)$ is a $P/Q$-polynomial digit expansion.
			To establish~\eqref{f2}, setting
			\begin{equation}\label{eq:prim}
			w'\vcentcolon=f\set{\smash{f_{\mathbf{Q}}^{-1}}\pdb(w)}_b=\pdb(w)-f\floor{\smash{f_{\mathbf{Q}}^{-1}}\pdb(w)}_b
			\end{equation}
			we have to show that $w'=w$. First note that by Lemma~\ref{lem:cut}~(ii)
			\[
			\deg_b\lt(f\set{\smash{f_\mathbf{Q}^{-1}\pdb(w)}}_b\rt)\leq 0
			\quad\text{and}\quad 
			\pdb(w)-f\floor{\smash{f_{\mathbf{Q}}^{-1}}\pdb(w)}_b\in\F_q[b]((X^{-1})),
			\]
			which implies that, $w'\in\F_q((X^{-1}))$. Since $\pdb(w) \in \F_q((b^{-1},X^{-1}))$, applying Lemma~\ref{lem:equiv}~(ii) to \eqref{eq:prim} yields 
			\begin{equation}\label{eq:wprimew}
			w'=\operatorname{eval_{d}}(w') = \operatorname{eval_{d}}(\pdb(w)).
			\end{equation}
			 Observe that $\operatorname{eval_{pd}}$ is just the restriction of $\operatorname{eval_{d}}$ to $\F_q[b](X)$. 
			Thus \eqref{eq:wprimew} implies that $w'=\operatorname{eval_{pd}}(\pdb(w))=w$ which completes the proof.
		\end{proof}		

		The analogue of Theorem~\ref{digitpoly} for any formal Laurent series over $\F_q$ in the disk $D(0,1)$ reads as follows.
		\begin{theorem}\label{digitLaurent}
			Let $(P/Q,\mathcal{D})$ be a rational function based digit system and set $f=Qb-P$.
			If $\alpha\in\F_q((X^{-1}))$ with $|\alpha|<1$, then
			\begin{align}
				\label{f3}
				\db(\alpha)&=\tfrac{f}{Q}\floor{\smash{f_{\mathbf{Q}}^{-1}}Q\alpha}_X=
				\alpha-\tfrac{f}{Q}\set{\smash{f_{\mathbf{Q}}^{-1}}Q\alpha}_X,\text{ and}\\
				\label{f4}
				\alpha&=f\floor{f_\mathbf{P}^{-1} \db(\alpha)}_b=\db(\alpha)-f\set{f_{\mathbf{P}}^{-1} \db(\alpha)}_b.
			\end{align}
		\end{theorem}
		\begin{proof}
			Define 
			\[\psi\vcentcolon=\tfrac{f}{Q}\floor{\smash{f_{\mathbf{Q}}^{-1}}Q\alpha}_X=\alpha-\tfrac{f}{Q}\set{\smash{f_{\mathbf{Q}}^{-1}}Q\alpha}_X.\]
			Then 
			\[\alpha-\psi=\tfrac{f}{Q}\set{\smash{f_{\mathbf{Q}}^{-1}}Q\alpha}_X.\]
			Since the form of $f_{\mathbf{Q}}^{-1}$ in~\eqref{invQ} implies that $(1/Q)\set{\smash{f_{\mathbf{Q}}^{-1}}Q\alpha}_X\in\F_q((b^{-1},X^{-1}))$, 
			Lemma~\ref{lem:equiv}~(ii) yields that $\operatorname{eval_d}(\psi)=\operatorname{eval_d}(\alpha)=\alpha$.
			
			It remains to show that $\psi$ is a $P/Q$-digit expansion. 
			For this purpose we first show that 
			\begin{equation}\label{eq:Qa}
				Q\psi=\sum_{i=-\infty}^k{s_ib^i}
			\end{equation} 
			with $s_i\in\F_q[X]$ and $\deg{s_i}<\deg{P}$. Indeed this follows from
			\[
			\deg_X({Q\psi})=\deg_X\lt(Q\alpha-f\set{\smash{f_{\mathbf{Q}}^{-1}}Q\alpha}_X\rt)<\deg{P}
			\]
			and
			\[
				Q\psi = f\floor{\smash{f_{\mathbf{Q}}^{-1}}Q\alpha}_X\in\F_q((b^{-1}))[X].
			\]
			In what follows, $g=O(b^{-\ell})$ means that $g\in\F_q((b^{-1},X^{-1}))$ with $\deg_b{g}\le -\ell$. 
			Since $f f_{\mathbf{Q}}^{-1}=1$ and $\deg_b(f_{\mathbf{Q}}^{-1})\le -1$ we must have 
			\begin{alignat*}{2}
				f_{\mathbf{Q}}^{-1}=\tfrac{1}{Q}b^{-1} + O(b^{-2})\;&\Longleftrightarrow\;& f_{\mathbf{Q}}^{-1}Q\alpha&=\alpha b^{-1}+O(b^{-2})\\ 
				& \Longleftrightarrow\; &\floor{\smash{f_{\mathbf{Q}}^{-1}}Q\alpha}_X&=\floor{O(b^{-2})}_X
			\end{alignat*}
			(note that the last line follows because $|\alpha| < 1$).
			Hence, $\deg_b{\floor{\smash{f_{\mathbf{Q}}^{-1}}Q\alpha}_X}\leq -2$ and so $\deg_b\lt(f\floor{\smash{f_{\mathbf{Q}}^{-1}}Q\alpha}_X\rt)\leq -1$. 
			Thus, we can strengthen \eqref{eq:Qa} by writing 
			\begin{equation}
				\begin{aligned}
					\floor{\smash{f_{\mathbf{Q}}^{-1}}Q\alpha}_X&=\sum_{i=1}^{\infty}v_{-i}b^{-i},\text{ and}\\			
					Q\psi= f\floor{\smash{f_{\mathbf{Q}}^{-1}}Q\alpha}_X&=\sum_{i=1}^{\infty}s_{-i}b^{-i},
				\end{aligned}\label{digitstree}
			\end{equation}
			where $v_{-i},s_{-i}\in\F_q[X]$, and $v_{-1}=0$.
			Comparing coefficients in~\eqref{digitstree}, we obtain $s_{-i}=Qv_{-i-1}-Pv_{-i}$ for all $i\in\N$.  
			Induction yields that for every positive integer $k$, it holds that 
			\[\frac{s_{-1}}{Q}{\lt(\frac{P}{Q}\rt)}^{k-1}+\frac{s_{-2}}{Q}{\lt(\frac{P}{Q}\rt)}^{k-2}+\cdots+\frac{s_{-k}}{Q}=v_{k+1}\in\F_q[X].\]
			Hence, $s_{-1}s_{-2}\cdots\in W_{P/Q}$ and applying Lemma~\ref{lem:72} completes the proof of~\eqref{f3}.
						
			We now proceed to show~\eqref{f4}.
			Assume that $\db(\alpha)$ corresponds to a $P/Q$-digit expansion of some $\alpha\in D(0,1)$.
			Let
			\[\alpha'\vcentcolon=f\floor{f_\mathbf{P}^{-1} \db(\alpha)}_b=\db(\alpha)-f\set{f_{\mathbf{P}}^{-1} \db(\alpha)}_b.\]
			By Lemma~\ref{lem:equiv}, we have to show that $\alpha'\in\F_q((X^{-1}))$. However, since
			\[
			\deg_b\lt(\db(\alpha)-f\set{f_{\mathbf{P}}^{-1} \db(\alpha)}_b\rt)\leq 0
			\quad\text{and}\quad
			f\floor{f_\mathbf{P}^{-1} \db(\alpha)}_b \in \F_q[b]((X^{-1})),\]
			we conclude that no nonzero power of $b$ occurs in $\alpha'$.  
			Hence, $\alpha'\in\F_q((X^{-1}))$.
		\end{proof}							
			
		\begin{remark}
			Note that the coefficients $v_{-i}$ of $\floor{\smash{f_{\mathbf{Q}}^{-1}}Q\alpha}_X$ in~\eqref{digitstree} have unbounded degree 
			and are precisely the $v_i$ (with a change of sign in the index) that are defined in~\eqref{rec2}.
		\end{remark}
		
		\begin{remark}\label{rem:k}
			For general $\alpha\in\F_q((X^{-1}))$, there exists $k\in\N$ such that $(P/Q)^{-k}\alpha\in D(0,1)$. 
			Then, since 
			\[\db(\alpha)=b^k\db\bigg({\lt(\frac{P}{Q}\rt)}^{-k}\alpha\bigg)\]
			we can still use Theorem~\ref{digitLaurent} in this case.
		\end{remark}
			
		\begin{remark}
			We mention that our unicity theorems as well as Theorems~\ref{digitpoly} and~\ref{digitLaurent} remain valid even if $\F_q$ is replaced by an arbitrary field $\F$. 
			However, since this would mean that our digit sets become infinite and thus all our automaticity results cannot be carried over to this setting 
			we decided to formulate all our results for the case of finite fields.
		\end{remark}			
			
	\section{Christol's Theorem for $P/Q$-digit systems}\label{sec:7}				
		In this section we will prove a version of the Theorem of Christol for rational function based digit systems 
		({\it cf}.~\cite{Christol1979,Christol1980} for the classical Theorem of Christol and~\cite[Theorem 12.2.5]{AS} for a proof of it in a more modern notation).  
		We first need to define the notion of $p$-automaticity for elements of $\F_q((\A))$. 
		For this we consider a two-dimensional sequence $(a_{mn})_{m,n\in\Z}$. 
		We denote the four quadrants of the plane by 
		\begin{alignat*}{4}
			\mathbf{Q}_1\vcentcolon&=\set{(x,y)\in\Z^2\,:\,x\geq 0, y\geq 0}, &\quad \mathbf{Q}_3\vcentcolon&=\set{(x,y)\in\Z^2\,:\,x\leq 0, y\leq 0}\\
			\mathbf{Q}_2\vcentcolon&=\set{(x,y)\in\Z^2\,:\,x\leq 0, y\geq 0}, &\quad \mathbf{Q}_4\vcentcolon&=\set{(x,y)\in\Z^2\,:\,x\geq 0, y\leq 0}.
		\end{alignat*}
		Then the sequence $(a_{mn})_{m,n\in\Z}$ is said to be $p$-automatic 
		if for each $i\in\set{1,2,3,4}$, there exists a DFAO that gives $a_{mn}$, where $(m,n)\in\mathbf{Q}_i$, 
		as output. Here, the input is the string of ordered pairs $(m_0,n_0)\cdots (m_k,n_k)$, where
		\[|m|=\sum_{i=0}^k{m_ip^i}\quad\text{and}\quad |n|=\sum_{i=0}^k{n_ip^i}\]
		with $m_i,n_i\in\set{0,\ldots,p-1}$ for $0\leq i\leq k$.
		
		We say that the $\A$-Laurent series $\alpha$ of the form~\eqref{ALaurent} is \emph{$p$-automatic} whenever 
		the two-dimensional sequence $(a_{ij})_{i,j\in\Z}$ formed by its coefficients is $p$-automatic.
		It is clear from the definition that $\alpha$ is $p$-automatic exactly when $\floor{\alpha}_{b}$ and $\set{\alpha}_{b}$
		(respectively, $\floor{\alpha}_{X}$ and $\set{\alpha}_{X}$) are $p$-automatic.
		
		We say that $\alpha\in\F_q((X^{-1}))$ is \emph{algebraic over $\F_q[X]$} if $s_0+s_1\alpha+\cdots+s_n\alpha^n=0$ 
		for some $s_0,s_1\ldots,s_n\in \F_q[X]$ which are not all equal to $0$.
		Similarly, a series $\alpha\in\F_q((\A))$ is \emph{algebraic over $\F_q[b,X]$} if $f_0+f_1\alpha+\cdots+f_n\alpha^n=0$
		for some $f_0,f_1\ldots,f_n\in\F_q[b,X]$ which are not all equal to $0$.		
		
		We need the following version of Christol's Theorem for $\F_q((\A))$.
		
		\begin{lemma}[{{\it cf.}~\cite[Proposition 5.12]{SS}}]\label{lem:2d}
			Let $q=p^s$ and $\A$ as in~\eqref{def:A}. 
			Then $\alpha\in\F_q((\A))$ is $p$-automatic if and only if $\alpha$ is algebraic over $\F_q[b,X]$.
		\end{lemma}
					
		We are now ready to state the main result of this section.
		Observe that each computational step in~\eqref{f3} and~\eqref{f4} is either automatic or algebraic.
		Theorem~\ref{Christol} then follows by switching back and forth between both notions and applying Lemma~\ref{lem:2d}.
		
		\begin{theorem}\label{Christol}
			Let $(P/Q,\mathcal{D})$ be a rational function based digit system, $q=p^s$ and $\alpha\in\F_q((X^{-1}))$. 
			Then $\db(\alpha)$ is $p$-automatic if and only if $\alpha$ is algebraic over $\F_q[X]$.
		\end{theorem}
		\begin{proof}
			If $\db(\alpha)$ is $p$-automatic, then $\db(\alpha)$ is algebraic over $\F_q[b,X]$ by Lemma~\ref{lem:2d}. 
			Since $f\in\F_q[b,X]$, it follows that $f_{\mathbf{P}}^{-1} \db(\alpha)$ is also algebraic. 
			Hence, $f_{\mathbf{P}}^{-1}\db(\alpha)$ is $p$-automatic by Lemma~\ref{lem:2d}. 
			Thus, $\floor{f_{\mathbf{P}}^{-1} \db(\alpha)}_b$ is $p$-automatic and so it is algebraic. 
			Finally, by~\eqref{f4}, we obtain that $\alpha=f\floor{f_{\mathbf{P}}^{-1} \db(\alpha)}_b$ is algebraic over $\F_q[b,X]$ by Lemma~\ref{lem:2d}. 
			Since $\alpha$ does not depend on $b$ it is even algebraic over $\F_q[X]$.
			
			Conversely, if $\alpha$ is algebraic over $\F_q[X]$ then $f_{\mathbf{Q}}^{-1}Q\alpha$ is algebraic over $\F_q[b,X]$. 
			This means that $\smash{f_{\mathbf{Q}}^{-1}}Q\alpha$ is $p$-automatic by Lemma~\ref{lem:2d}. 
			Thus, $\floor{\smash{f_{\mathbf{Q}}^{-1}}Q\alpha}_{X}$ is also $p$-automatic and, hence, algebraic by Lemma~\ref{lem:2d}. 
			It now follows from~\eqref{f3} that $\db(\alpha)=(f/Q)\floor{\smash{f_{\mathbf{Q}}^{-1}}Q\alpha}_X$ is algebraic and, hence, $p$-automatic by Lemma~\ref{lem:2d}.
	\end{proof}	
	
	\section{Connection to Mahler's Problem for finite fields}\label{sec:6}
		A well-known unsolved problem in number theory is the quest for the distribution of the sequence $((p/q)^i \wmod{1})_{i\in\N}$, where $p,q$ are coprime integers with $p>q\geq 2$.  
		The analogous problem for polynomials over finite fields was considered in~\cite{ADKK:01}.  
		In particular, if $f=P/Q$, where $P,Q\in\F_q[X]$ are coprime with $\deg P>\deg Q \ge 1$, then it was shown that for specific cases that
		the distribution of the sequence $(\set{f^i})_{i\in\N}$ is the Dirac measure $\delta_0$ at $0\in\F_q((X^{-1}))$ or is continuous.
		
		A reformulation of this problem by Mahler is to study instead the distribution of the sequence $(z{(p/q)}^i\wmod{1})_{i\in\N}$ for different real values of $z$.  
		Using the $p/q$-digit expansion of real numbers, it was established in~\cite{AFS} that if $p\geq 2q-1$, 
		then there exists a subset $Y_{p/q}$ of $\R/\Z$ with Lebesgue measure $q/p$ such that
		\[\{z\in\R\,:\,\text{ there is }i_0\in\N \text{ such that } \,z{(p/q)}^i \wmod{1}\in Y_{p/q}\text{ for }i\ge i_0\}\]
		is countable infinite.  
		
		We give here the analogous result for the finite fields case.
		Denote by $w(v)$ the label string of an infinite path that starts with the node $v$ in $T(P/Q)$ 
		and that always follows the edges whose labels are digits of degree less than $\deg{Q}$.  
		By~\eqref{edgelabels}, $w(v)$ is unique for each $v\in\F_q[X]$.  
		We call these strings \emph{minimal strings}.  
		An element of $W_{P/Q}$ is said to be \emph{eventually minimal} if it has a suffix which is a minimal string.  
		Clearly, as $\F_q[X]$ is countable, there are countably infinitely many eventually minimal elements of $W_{P/Q}$.
		
		\begin{theorem}\label{mahler}
			Let $(P/Q,\mathcal{D})$ be a rational function based digit system and set $m=\deg P$, $n=\deg Q$ and $r=m-n$.
			If \[Y_{P/Q}=\bigcup_{i=1}^{\infty}[q^{-ri-n},q^{-ri}),\]
			then the set 
			\[\{\alpha\in\F_q((X^{-1}))\,:\,\text{there is }i_0\in\N \text{ such that } |\{\alpha{(P/Q)}^i\}| \in Y_{P/Q}\cup\set{0}\text{ for }i\ge i_0\}\]
			consists of all $\alpha\in\F_q((X^{-1}))$ such that the digit string associated with the $P/Q$-digit expansion of $\alpha$ is eventually minimal.
		\end{theorem}
		\begin{proof}
			Let $\alpha\in\F_q((X^{-1}))$ with $\sexpan{\alpha}_{P/Q}=s_k\cdots s_0.s_{-1}s_{-2}\cdots$.
			Since \[s_k\cdots s_0s_{-1}s_{-2}\cdots\in W_{P/Q},\] we have $\{\alpha{(P/Q)}^i\}=\pi(.s_{-i-1}s_{-i-2}\cdots)$.
			Then it follows from Theorem~\ref{eventperiod} that $|\{\alpha{(P/Q)}^i\}|$ is eventually $0$ if and only if $\alpha=0$.
			
			Suppose $\alpha\neq 0$.  
			If $s_k\cdots s_0s_{-1}s_{-2}\cdots$ is eventually minimal, then there exists $\ell\in\N$ such that $\deg{s_{-i}}<n$ for $i>\ell$.  
			Thus, by Lemma~\ref{degreeint} and Theorem~\ref{eventperiod}, we have $|\{\alpha{(P/Q)}^i\}|\in Y_{P/Q}$ for all $i>\ell$.			
			On the other hand, if $s_k\cdots s_0s_{-1}s_{-2}\cdots$ is not eventually minimal, 
			then we can always find a sufficiently large $\ell\in\N$ such that $n\leq \deg{s_{-\ell}}<m$.  
			Then by Lemma~\ref{degreeint}, $|\{\alpha{(P/Q)}^{\ell}\}|\notin Y_{P/Q}$.
		\end{proof}		
		
		\begin{remark}
			If $m\leq 2n$, then $Y_{P/Q}$ simplifies to $Y_{P/Q}=(0,q^{-r})$.
		\end{remark}
	
		\subsection*{Acknowledgement} 
			The authors thank Shigeki Akiyama for helpful discussions on topics related to the present paper. 
			Furthermore, we are indebted to the anonymous referee for his careful reading of the manuscript.

\end{document}